\documentclass[12pt, reqno]{amsart}
\usepackage{amsmath,amssymb,amsthm}

\usepackage{dsfont}
\usepackage{amsmath, amssymb}
\usepackage{amsthm, amsfonts, mathrsfs}
\usepackage{mathptmx}
\usepackage{fullpage}
\usepackage{amsfonts,graphicx}
\numberwithin{equation}{section}
\usepackage[colorlinks=true, pdfstartview=FitV, linkcolor=blue, citecolor=blue, urlcolor=blue]{hyperref}

\newcommand\blfootnote[1]{%
  \begingroup
  \renewcommand\thefootnote{}\footnote{#1}%
  \addtocounter{footnote}{-1}%
  \endgroup}

\newcommand{\NN}{\mathbb{N}}

\newcommand{\RR}{\mathbb{R}}

\newcommand{\EE}{\varepsilon}

\newcommand{\Ss}{\mathbb{S}}

\newcommand{\Div}{\textnormal{div}}

\newcommand{\supp}{\textnormal{supp }}

\newcommand{\Curl}{\textnormal{curl}}
\newcommand{\essup}{\textnormal{essup}}
\newcommand{\lin}{\textnormal{lin}} 
\newtheorem{Theo}{Theorem}[section]

\newtheorem{prop}[Theo]{Proposition}

\theoremstyle{plain}
\theoremstyle{definition}

\theoremstyle{remark}
\newtheorem{Rema}[Theo]{Remark}
\newtheorem*{rema*}{Remark}

\author[A. Hanachi]{Adalet Hanachi}
\address{LEDPA, Universit\'e de Batna --2--\\ Facult\'e des Math\'ematiques et Informatique\\ D\'epartement de Math\'ematiques\\ 05000 Batna Alg\'erie}
\email{a.hanachi@univ-batna2.dz}

\author[H. Houamed]{Haroune Houamed}
\address{CNRS, LJAD, Université Co\^te d'Azur\\ Département de Mathématiques\\ Nice,  France}
\email{haroune.houamed@univ-cotedazur.fr}

\author[M. Zerguine]{Mohamed Zerguine}
\address{LEDPA, Universit\'e de Batna --2--\\ Facult\'e des Math\'ematiques et Informatique\\ D\'epartement de Math\'ematiques\\ 05000 Batna Alg\'erie}
\email{m.zerguine@univ-batna2.dz}

\date{}

\begin{document}

\title[Inviscid limit]{On the global well-posedness of the axisymmetric viscous Boussinesq system in critical Lebesgue spaces}
\maketitle
\begin{abstract} The contribution of this paper will be focused on the global existence and uniqueness topic in three-dimensional case of the axisymmetric viscous Boussinesq system in critical Lebesgue spaces. We aim at deriving analogous results for the classical two-dimensional and three-dimensional axisymmetric Navier-Stokes equations recently obtained in \cite{Gallay,Gallay-Sverak}. Roughly speaking, we show essentially that if the initial data $(v_0,\rho_0)$ is axisymmetric and $(\omega_0,\rho_0)$ belongs to the critical space $L^1(\Omega)\times L^1(\RR^3)$, with  $\omega_0$ is the initial vorticity associated to $v_0$ and $\Omega=\{(r,z)\in\RR^2:r>0\}$, then the viscous Boussinesq system has a unique global solution. 

\end{abstract}
\noindent 
\blfootnote{{\it keywords and phrases:}
Boussinesq sytem; Axisymmetric solutions; Critical spaces; Global well-posedness.}  
\blfootnote{2010 MSC: 76D03, 76D05, 35B33, 35Q35.}
\tableofcontents
\section{Introduction} 
\hspace{0.7cm} The description of the state of a moving stratified fluid in three-dimensional under the Boussinesq approach, taking into account the friction forces is determined by the distribution of the fluid velocity $v(t,x)$ with free-divergence located in position $x$ at a time $t$, the scalar function $\rho(t,x)$ designates either the temperature in the context of thermal convection, or the mass density in the modeling of geophysical fluids and $p(t,x)$ is the pressure which is relates $v$ and $\rho$ through an elliptic equation. This provides us to the Cauchy problem for the Boussinesq system,  
\begin{equation}\label{B(mu,kappa)}
\left\{ \begin{array}{ll} 
\partial_{t}v+v\cdot\nabla v-\mu\Delta v+\nabla p=\rho \vec e_z & \textrm{if $(t,x)\in \RR_+\times\RR^3$,}\\
\partial_{t}\rho+v\cdot\nabla \rho-\kappa\Delta \rho=0 & \textrm{if $(t,x)\in \RR_+\times\RR^3$,}\\ 
\Div v=0, &\\ 
({v},{\rho})_{| t=0}=({v}_0,{\rho}_0). \tag{B$_{\mu,\kappa}$}
\end{array} \right.
\end{equation} 
Above, $\mu$ and $\kappa$ are two nonnegative parameters which can be seen as the inverse of Reynolds numbers and $\rho\vec{e}_z$ models the influence of the buoyancy force in the fluid motion in the vertical direction  $\vec{e}_z=(0,0,1)$. 

\hspace{0.7cm}Boussinesq flows are ubiquitous in various nature phenomenon, such as oceanic circulations, atmospheric fronts or katabatic winds, industry such as fume cupboard ventilation or dense gas dispersion, see, e.g. \cite{Ped}.

\hspace{0.7cm} Let us notice that the Navier-Stokes equations are obtained as particular case from \eqref{B(mu,kappa)} when the initial density is constant. Such equations are given as follows
\begin{equation}\label{NS(mu)}
\left\{ \begin{array}{ll} 
\partial_{t}v+v\cdot\nabla v-\mu\Delta v+\nabla p=0 & \textrm{if $(t,x)\in \RR_+\times\RR^3$,}\\
\Div v=0, &\\ 
v_{| t=0}={v}_0. \tag{NS$_\mu$}
\end{array} \right.
\end{equation}

\hspace{0.7cm}An important breakthrough was J. Leray's paper in the thirties of last century where he showed the global existence of weak solutions in energy space for any dimension. Nevertheless, the uniqueness of such solutions has been till now an open question, unless for the two-dimensional case. Lately, the local well-posedness issue in the setting of mild solutions for \eqref{NS(mu)} was done by H. Fujita and T. Kato \cite{fk} for initial data belonging to the critical Sobolev space $\dot{H}^{\frac{1}{2}}$ in the sense of {\it scale invariance}.  More similar results are established in several functional spaces like $L^3,$  $\dot{{B}}_{p, \infty}^{-1+\frac{3}{p}}$ and $BMO^{-1}$. It should be noted that these types of solutions are globally well-posed in a time for initial data sufficiently small with respect to the viscosity, except in two-dimensional, see  \cite{LemarBook,Robinson-Rodrigo-Sadowski}. In a similar way, especially in dimension two of spaces, the system \eqref{B(mu,kappa)} was tackled by enormemos authors in various functional spaces and different values for the parameters $\kappa$ and $\mu$. For the connected subject, we refer to some selected references \cite{hk1, HKR1, HKR2, HZ, Tit, MX, ES}. 

\hspace{0.7cm} Before discussing some theoretical underpinnings results on the well-posedness topic for the viscous Boussinesq system \eqref{B(mu,kappa)} in three-dimensional. First, let us notice that the topic of global existence and uniqueness for \eqref{NS(mu)} in the general case is a till an open problem in PDEs. It is therefore incumbent upon us to seek a subclass of vector fields which in turns leading to some conservation quantities, and so the global well-posedness result. Such subclass involving to rewriting \eqref{NS(mu)} under vorticity formulation by applying the "$\Curl$" to the momentum equation, which is defined by $\omega=\nabla \times   v$. Thus, we have: 
\begin{equation}
\left\{ \begin{array}{ll} 
\partial_{t}\omega+v\cdot\nabla \omega-\mu\Delta \omega =\omega\cdot\nabla v & \textrm{if $(t,x)\in \RR_+\times\RR^3$,}\\
\omega_{| t=0}={\omega}_0.
\end{array} \right.
\end{equation} 
According to Beale-Kato-Majda criterion in \cite{BKM}, the control $\omega$ in $L_t^1 L^\infty$ is a key step for the global well-posedness of the solutions of \eqref{NS(mu)}. However, for three-dimensional flow the situation is more complicated due to the presence of stretching term $\omega\cdot\nabla v$, which contributes additional drawbacks for the fluid motion. 

\hspace{0.7cm}The looked subclass requires to assume that the velocity vector field is an axisymmetric vector field without {\it swirl} in the sense that $v$ can be decomposed in the cylindrical coordinates $(r,\theta,z)$ as follows: 
\begin{equation*}
v(t, x)=v^{r}(t, r, z)\vec{e}_{r}+v^{z}(t, r, z)\vec{e}_{z},
\end{equation*}
where for every $x=(x_1,x_2,z)\in\RR^3$ we have
\begin{equation*}
x_1=r\cos\theta,\quad x_2=r\sin\theta,\quad r\ge0,\quad 0\le\theta<2\pi. 
\end{equation*}
Above, the triplet  $(\vec{e}_{r}, \vec{e}_{\theta}, \vec{e}_{z})$ represents the usual frame of unit vectors in the radial, azimuthal and vertical directions with the notation
\begin{equation*}
\vec{e}_r=\Big(\frac{x_1}{r},\frac{x_2}{r},0\Big),\quad \vec{e}_{\theta}=\Big(-\frac{x_2}{r},\frac{x_1}{r},0\Big),\quad \vec{e}_{z}=(0,0,1).
\end{equation*}
For these flows the vorticity $\omega$ takes the form $\omega\triangleq\omega_{\theta}\vec{e}_{\theta}$, with  
\begin{equation}\label{component(theta)}
\omega_{\theta}=\partial_{z}v^r-\partial_{r}v^z.
\end{equation}
Taking advantage to $\Div v=0$, the velocity field can be determined clearly in the half-space $\Omega=\{(r,z)\in\RR^2:r>0\}$ by solving the following elliptic system
\begin{equation}\label{BS}
\left\{ \begin{array}{ll} 
\partial_{r}v^{r}+\frac1r v^r+\partial_z v^z=0, & \\
\partial_{z}v^r-\partial_{r}v^z=\omega_{\theta}, 
\end{array} \right.
\end{equation}
under homogeneous boundary conditions $v^{r}=\partial_{r}v^{z}=0$. The differential system \eqref{BS} is well-known the {\it axisymmetric Biot-Savart} law associated to \eqref{NS(mu)}, see Section \ref{Biot-Savart law}.

\hspace{0.7cm} A few computations claim that the streatching term $\omega\cdot\nabla v$ close to $\frac{v^r}{r}\omega_{\theta}$  and that $\omega_{\theta}$ evolves,
\begin{equation}\label{omega(theta)}
\partial_{t}\omega_{\theta}+(v\cdot\nabla)\omega_{\theta}-\mu\Delta\omega_{\theta}+\mu\frac{\omega_{\theta}}{r^2}=\frac{v^r}{r}\omega_{\theta},
\end{equation}
with the notations $v\cdot\nabla=v^{r}\partial_{r}+v^{z}\partial_{z}$ and $\Delta=\partial_{r}^{2}+\frac1r\partial_{r}+\partial_{z}^{2}$. By setting $\Pi=\frac{\omega_{\theta}}{r}$, we discover that $\Pi$ satisfies
\begin{equation}\label{zeta}
\partial_{t}\Pi+v\cdot\nabla \Pi-\mu\Big(\Delta +\frac{2}{r}\partial_r\Big)\Pi=0,\quad\Pi_{| t=0}=\Pi_0.  
\end{equation}
Since the dissipative operator $(\Delta +\frac{2}{r}\partial_r)$ has a good sign, thus the $L^p-$norms are time bounded, that is for $t\ge0$
\begin{equation}\label{Pi}
\|\Pi(t)\|_{L^p}\le \|\Pi_0\|_{L^p}, \quad p\in[1,\infty]. 
\end{equation}
Under this pattern,  M. Ukhoviskii and V. Yudovich \cite{uy}, independently O. Ladyzhenskaya \cite{ol} succeed to recover \eqref{NS(mu)} globally in time, whenever  $v_0\in H^1$ and $\omega_0,\frac{\omega_0}{r}\in L^2\cap L^\infty$. This result was relaxed later by S. Leonardi, J. M\`alek, J. Nec\u as and M. Pokorn\'y for $v_0\in H^2$ and weakened recently in \cite{a} by H. Abidi  for $v_0\in H^{\frac{1}{2}}$. 

\hspace{0.7cm}The majority of aforementioned results are accomplished within the framework of finite energy solutions. For the solutions with infinite energy, in particular in two dimensions many results dealing with the global well-posedness problem have been obtained by numerous authors. Particularly, worth mentioning that Giga, Miyakawa and Osada have been established in \cite{giga-miyakawa-osada} that \eqref{NS(mu)} admits a unique global solutions for initial vorticity is measure. Lately, M. Ben-Artzi \cite{Ben-Artzi} has shown that  \eqref{NS(mu)} is globally well-posed whenever the initial vorticity $\omega_0$ belongs to critical Lebesgue space $L^1(\RR^2)$ who proposed a new formalism based on elementary comparison principles for linear parabolic equations. While, the uniqueness was relaxed later by H. Brezis \cite{Brezis}. Thereafter, this result was  improved by I. Gallagher and Th. Gallay \cite{Gallagher-Gallay}, where they constructed a solutions globally in time, under the assumption that $\omega_0$ is a finite measure. For more details about this subject we refer the reader to the references \cite{Cottet,Gallay,PG}.

\hspace{0.7cm}More recently, the global well-posedness problem for \eqref{omega(theta)} was revisited in three-dimensional by Th. Gallay and V. \u Sver\'ak who established in \cite{Gallay-Sverak} that \eqref{NS(mu)} posseses a unique global solutions if the initial velocity is an axisymmetric vector field and its vorticity lying the critical space $L^1(\Omega)$. In addition, they were extending their results in more general case, i.e., $\omega_0$ is a finite measure, where $\Omega$ is half-plane $\Omega=\{(r,z)\in\RR^2: r>0,\;z\in\RR\}$ endowed with the product measure $drdz$. Actually, their paradigm uses specifically the standard fixed point method to show in particular the local well-posedness for the system \eqref{omega(theta)} under the form
\begin{equation*}
\partial_{t}\omega_{\theta}+\Div_{\star}(v\omega_{\theta})=\mu\Big(\Delta-\frac{1}{r^2}\Big)\omega_{\theta}
\end{equation*}
combined with the special structure of the vorticity, in particular, the axisymmetric Biot-Svart law, where $\Div_{\star}f=\partial_r f^r+\partial_z f^z$ in general case. They showed that the local solutions often constructed can be extended to the global one by exploiting some a priori estimates for $\omega_{\theta}$ in various norms. We point us that the uniqueness topic for initial vorticity is measure was done by the same authors under some smallness condition for the ponctual part. Furthermore, they provide also an asymptotic behavior study for positive vorticity with a finite impulse.

\hspace{0.7cm}Apropos of \eqref{B(mu,kappa)}, the global regularity in dimension three of spaces has received a considerable attention. As shown in \cite{Danchin-Paicu}, for $\kappa=0$ R. Danchin and M. Paicu investigated that \eqref{B(mu,kappa)} is well-posed in time for any dimension in the framework of Leray's and Fujita-Kato's solutions, excpect in three-dimensional and under also smallness condition, the system \eqref{B(mu,kappa)} is also global. Next, in the axisymmetric case, H. Abidi, T. Hmidi and S. Keraani were establishing in \cite{Abidi-Hmidi-Keraani} that \eqref{B(mu,kappa)} is globally well-posed by rewriting it under vorticity-density formulation:
\begin{equation}\label{VD11}
\left\{ \begin{array}{ll} 
\partial_{t}\omega_{\theta}+v\cdot\nabla \omega_{\theta}-\frac{v^r}{r}\omega_{\theta}=\big(\Delta-\frac{1}{r^2}\big)\omega_{\theta}-\partial_r\rho, & \\
\partial_{t}\rho+v\cdot\nabla \rho=0, & \\ 
(\omega_\theta,{\rho})_{| t=0}=({\omega}_0,{\rho}_0). 
\end{array} \right.
\end{equation}
Consequently, the quantity $\Pi=\frac{\omega_{\theta}}{r}$ solving the equation
\begin{equation}\label{Pi-Boussinesq}
\partial_{t}\Pi+v\cdot\nabla \Pi-(\Delta +\frac{2}{r}\partial_r)\Pi=-\frac{\partial_r\rho}{r}.
\end{equation}   

They assumed that $v_0\in H^1(\RR^3),\;\Pi_0\in L^{2}(\RR^3), \rho_0\in L^2\cap L^\infty$ with $\supp\rho_{0}\cap(Oz)=\emptyset$ and $P_{z}(\supp\rho_0)$ is a compact set in $\RR^3$, especially to dismiss the violent singularity of $\frac{\partial_r\rho}{r}$, with $P_{z}$ being the orthogonal projector over $(Oz)$. Those results are improved later by T. Hmidi and F. Rousset in \cite{Hmidi-Rousset} for $\kappa>0$ by removing the assumption on the support of the density. Their strategy is deeply based on the coupling between the two equations of the system \eqref{VD} by introducing a new unknown which called {\it coupled function} (see, Section \ref{Setup and preliminary results}). In the same way, more recently P. Dreyfuss and second author \cite{Dreyfuss-Houamed} treated the system in question, where they replaced the full dissipation by an horizontal one in all the equations, and proved the global wellposedness if the axisymmetric initial data $(v_0,\Pi_0,\rho_0)$ lies in $H^1(\RR^3) \times L^2(\RR^3) \times L^2(\RR^3)$ . In the same direction, in \cite{Houamed-Zerguine} the second and third authors, succeed to solve \eqref{B(mu,kappa)} globally in time for $\kappa=0$ and axisymmetric initial data $(v_0,\rho_0) \in \big( H^\frac{1}{2} \cap \dot{B}^0_{3,1} \big)(\RR^3) \times \big(L^2 \cap \dot{B}^0_{3,1} \big)(\RR^3)$, by essentially combining the works of\cite{a,Danchin-Paicu,Hmidi-Rousset}.\\

\hspace{0.7cm} In the present paper, we are interested to conduct the same results recently obtained in \cite{Gallay-Sverak} for the viscous Boussinesq system \eqref{B(mu,kappa)} expressed by the following vorticity-density formulation.
\begin{equation}\label{VD}
\left\{ \begin{array}{ll} 
\partial_{t}\omega_{\theta}+v\cdot\nabla \omega_{\theta}-\frac{v^r}{r}\omega_{\theta}=\big(\Delta-\frac{1}{r^2}\big)\omega_{\theta}-\partial_r\rho, & \\
\partial_{t}\rho+v\cdot\nabla \rho-\kappa\Delta \rho=0, & \\ 
(\omega_\theta,{\rho})_{| t=0}=({\omega}_0,{\rho}_0),  
\end{array} \right.
\end{equation}

for initial data $(\omega_0,\rho_0)$ in the critical Lebesgue space $L^1(\Omega)\times L^1(\RR^3)$, with the following notations 

\begin{equation*}
\|\omega_{\theta}\|_{L^p(\Omega)}=\left\{\begin{array}{ll}
\Big(\int_{\Omega}|\omega_{\theta}(r,z)|^p drdz\Big)^{\frac1p} & \textrm{if $p\in [1,\infty)$,}\\
\essup_{(r,z)\in\Omega}|\omega_{\theta}(r,z)| & \textrm{if $p=\infty$}
\end{array}
\right.
\end{equation*}
and
\begin{equation*}
\|\Pi\|_{L^p(\RR^3)}=\left\{\begin{array}{ll}
\Big(\int_{\Omega}|\Pi(r,z)|^p rdrdz\Big)^{\frac1p} & \textrm{if $p\in [1,\infty)$,}\\
\essup_{(r,z)\in\Omega}|\Pi(r,z)| & \textrm{if $p=\infty$}.
\end{array}
\right.
\end{equation*}
Let us denote that the spaces $L^1(\Omega)$ and $L^1(\RR^3)$ are scale invariant, in the sense
\begin{equation*}
\|\lambda^2\omega_{0}(\lambda \cdot)\|_{L^1(\Omega)}=\|\omega_0\|_{L^1(\Omega)},\quad \|\lambda^3\rho_0(\lambda\cdot)\|_{L^1(\RR^3)}=\|\rho_0\|_{L^1(\RR^3)}
\end{equation*}
for any $(\omega_0,\rho_0)\in L^1(\Omega)\times L^1(\RR^3)$ and any $\lambda>0$. This is derived from the fact 
\begin{equation*}
\omega_{\theta}(t,r,z)\mapsto \lambda^{2}\omega_{\theta}(\lambda^2 t, \lambda r,\lambda z),\quad \rho(t,r,z)\mapsto \lambda^3\rho(\lambda^2t,\lambda r,\lambda z). 
\end{equation*}
is a symmetry of \eqref{VD}.

\hspace{0.7cm}At this stage, we are ready to state the main result of this paper. To be precise, we will prove the following theorem.
\begin{Theo}\label{First-Th.} Let $(\omega_{0},\rho_{0})\in L^1(\Omega)\times L^1 (\RR^3)$ be an axisymmetric initial data, then the system \eqref{VD} for $\kappa=1$ admits a unique global mild solution. More precisely, we have:
\begin{equation}\label{first-result}
(\omega_{\theta}, r\rho) \in \Big(  C^0\big([0,\infty);  L^1(\Omega)\big)\cap C^0\big((0,\infty); L^\infty(\Omega)\big) \Big)^2,
\end{equation}
\begin{equation}\label{first-result2}
 \rho \in C^0\big([0,\infty);  L^1(\RR^3)\big)\cap C^0\big((0,\infty); L^\infty(\RR^3)\big).
\end{equation}
Furthermore, for every $p\in [1,\infty]$, there exists some constant $\tilde{K} _p(D_0)>0$, for which, and for all $t>0$ the following statements hold. 
\begin{equation}\label{boundedness}
\|(\omega_{\theta}(t),r\rho(t))\|_{L^p(\Omega)\times L^p(\Omega)}\leq t^{-(1-\frac{1}{p})}  \tilde{K}_p(D_0),
\end{equation}
\begin{equation}\label{boundedness2}
\| \rho(t) \|_{L^p(\RR^3) }\leq t^{-\frac{3}{2} (1-\frac{1}{p})}  \tilde{K}_p(D_0),
\end{equation}
where
\begin{equation*}
D_0 \triangleq\|(\omega_0,\rho_0)\|_{L^1(\Omega)\times L^1(\RR^3)}.
\end{equation*}
\end{Theo}
A few comments about the previous Theorem are given by the following remarks.
\begin{Rema} By axisymmetric scalar function we mean again a function that depends only on the variable $(r,z)$ but not on the angle variable $\theta$ in cylindrical coordinates. We check obviously that the axisymmetric structure is preserved through the time in the way that if $(v_0,\rho_0)$ is axisymmetric without swirl, then the obtained solution is it also. 
\end{Rema}

\begin{Rema} The hypothesis $\omega_{\theta}$ is in $ L^{1}(\Omega)$ doesn't implies generally that the associated velocity $v$ is in $L^2(\Omega)$ space. Consequently, the classical energy estimate is not available to derive a uniform bound for the velocity. 
\end{Rema}
\hspace{0.7cm} The proof is organized in two parts. The first one cares with the local well-posedness topic for \eqref{VD} in the spirit of Gallay and Sv\'erak \cite{Gallay-Sverak}. We make use of fixed point-method for an equivalent system \eqref{VD} on product space equipped with an adequate norm with the help of the axisymmetric Biot-Savart law and some norm estimates between the velocity and vorticity. But in our context, we should deal carefully wih the additional term $\frac{\partial_r\rho}{r}$ which contributes a singularity over the axe $(Oz)$. The remedy is to hide this term by exploiting the coupling structure of the system \eqref{VD} for $\kappa=1$ and introducing a new unknown functions $\Gamma$ and $\widetilde{\Gamma}$ in the spirit of \cite{Hmidi-Rousset} by setting $\Gamma=\frac{\omega_{\theta}}{r}-\frac{\rho}{2}$, and $\widetilde{\Gamma} = r \Gamma$. A straightforward computation shows that $\Gamma$ and $\widetilde{\Gamma}$ solve, respectively
\begin{equation}\label{zeta0}
\left\{ \begin{array}{ll} 
\partial_{t}\Gamma+v\cdot\nabla\Gamma-(\Delta +\frac{2}{r}\partial_r)\Gamma=0,\\
\Gamma_{t=0}=\Gamma_0.  
\end{array} \right.
\end{equation}
\begin{equation}\label{zeta0-tildegamma}
\left\{ \begin{array}{ll} 
\partial_{t}\widetilde{\Gamma}+v\cdot\nabla \widetilde{\Gamma} -\frac{v^r}{r} \widetilde{\Gamma} -(\Delta -\frac{1}{r^2} )\widetilde{\Gamma}=0,\\
\widetilde{\Gamma}_{t=0}=r\Gamma_0.  
\end{array} \right.
\end{equation}
In fact, in the second part, we shall investigate some a priori estimates for the different variables, in order to derive the global regularity for the system in question. Significant properties of the new unknowns, such as the maximum principle, are gained in this transition and $\Gamma$ evolves a similar equation and keeps the same boundary conditions than $\Pi$ in the case of the axisymmetric Navier-Stokes without swirl, see \eqref{zeta}. As consequence, the function $\Gamma$ (and eventually $\widetilde{\Gamma}$ after some technical computations) satisfies the boundedness estimate as in \eqref{Pi} which will be crucial in the process of deriving the global regularity of our solutions.  

\hspace{0.7cm} For the reader's convenience, we provide a brief headline of this article. In section 2, we briefly depict the framework that exists regarding the axisymmetric Biot-Savart law. Many results could be spent in explaining this framework in detail, in particular, the relation between the velocity vector field and its vorticity by means of stream function. Along the way, we recall some weighted estimates which will be helpful in the sequel. Afterwards, we focus in the linear equation of \eqref{VD} and some characterization of their associated semigroup, in particular the $L^p\rightarrow L^q$ estimate as in two-dimension space. Section 3,  mainly treats the local well-posedness topic for the system \eqref{VD}.  The main tool is the fixed point argument on the product space combined with a few technics about the semigroup estimates. 
In section 4, we investigate some global a priori estimates by coupling the system \eqref{VD} and introducing the new unknowns $\Gamma$ and $\widetilde{\Gamma}$. Considering these latest quantities will be a helpful to derive the global existence for the equivalent system \eqref{VD} and consequently the system (B$_{\mu,\kappa}$). 

\section{Setup and preliminary results}\label{Setup and preliminary results} In this section we recall some basic tools which will be employed in the subsequent sections. In particular, we devellop the Biot-Savart law in the framework of axisymmetric vector fields, and we study the linear equation associated to the system \eqref{VD}, usually specialized to the local existence.
\subsection{The tool box of Biot-Savart law}\label{Biot-Savart law} Recalling that in the cylindrical coordinates and in the class of axisymmetric vector fields without swirl the velocity is given by $v=(v^r,0,v^z)$ with $v^r$ and  $v^z$ are independtly of $\theta-$variable, $\omega_{\theta}$ its vorticity defined from $\Omega$ into $\RR$ by $\omega_{\theta}=\partial_z v^r-\partial_r v^z$ and the divergence-free condition $\Div v=0$ turns out to be 
\begin{equation*}
\partial_r(rv^r)+\partial_z(rv^z)=0.
\end{equation*}
In this case, it is not difficult to build a scalar function $\Omega\ni(r,z)\mapsto\psi(r,z)\in\RR$ which called {\it axisymmetric stream function} and satisfying
\begin{equation}\label{v(r),v(z)}
v^r=-\frac{1}{r}\partial_z\psi,\quad v^z=\frac{1}{r}\partial_r\psi.
\end{equation}  
Consequently, one obtains that $\psi$ evolves the following linear elliptic inhomogeneous equation
\begin{equation*}\label{psi-linear-elliptic}
-\frac{1}{r}\partial_r^2\psi+\frac{1}{r^2}\partial_{r}\Psi-\frac{1}{r}\partial^2_{z}\psi=\omega_{\theta},
\end{equation*}
with the boundary conditions $\psi(0,z)=\partial_{r}\psi(0,z)=0$. By setting $\mathcal{L}=-\frac{1}{r}\partial_r^2+\frac{1}{r^2}\partial_{r}-\frac{1}{r}\partial_{z}^2$, one finds the following boundary value problem
\begin{equation}\label{boundary-VP}
\left\{\begin{array}{ll}
\mathcal{L}\psi(r,z)=\omega_{\theta}(r,z) & \textrm{if $(r,z)\in\Omega$} \\
\psi(r,z)=\partial_{r}\psi(r,z)=0 & \textrm{if $(r,z)\in\partial\Omega$},
\end{array}
\right.
\end{equation}
where $\partial\Omega=\{(r,z)\in\RR^2: r=0\}$. It is evident that $\mathcal{L}$ is an elliptic operator of second order, then according to \cite{Sv}, $\mathcal{L}$ is invertible with an inverse $\mathcal{L}^{-1}$. Consequently, the last boundary value problem admits a unique solution given by
\begin{equation}\label{Psi}
\Psi(r,z)\triangleq\mathscr{L}^{-1}\omega_{\theta}(r,z)=\int_{-\infty}^{\infty}\int_{0}^{\infty}\frac{\sqrt{\widetilde{r}r}}{2\pi}F\Bigg(\frac{(r-\widetilde{r})^2+(z-\widetilde{z})^2}{\widetilde{r}r}\Bigg)\omega_{\theta}(\widetilde{r},\widetilde{z})d\widetilde{r}d\widetilde{z},
\end{equation}
where the function $F:(0,\infty)\rightarrow\RR$ is expressed as follows.
\begin{equation}\label{F-function}
F(s)=\int_{0}^{\pi}\frac{\cos\alpha d\alpha}{\big(2(1-\cos\alpha)+s\big)^{1/2}}.
\end{equation}
Since, $F$ cannot be expressed as an elementary functions, but it contributes some asymptotic properties near $s=0$ and $s=\infty$ listed in the following proposition. For more details about the proof, see \cite{Feng-Sverak,Sv}.
\begin{prop}\label{property-F} Let $F$ be the function defined in \eqref{F-function}, then the following assertions are hold.
\begin{itemize}
\item[{\bf(i)}] $F(s)=\frac12\log\frac{1}{s}+\log 8-2+O\Big(s\log\frac{1}{s}\Big)$ and $F'(s)=-\frac{1}{2s}+O\Big(\log\frac{1}{s}\Big)$ as $s\rightarrow 0^{+}$.
\item[{\bf(ii)}] $F(s)=\frac{\pi}{2s^{3/2}}+O\Big(\frac{1}{s^{5/2}}\Big)$ and $F'(s)=-\frac{3\pi}{4s^{5/2}}+O\Big(\frac{1}{s^{7/2}}\Big)$ as $s\rightarrow\infty$. 
\item[{\bf(iii)}] For every $k\in\NN^\star$, we have
\begin{equation*}
|F(s)|\lesssim\min\bigg(\Big(\frac{1}{s}\Big)^{\epsilon},\Big(\frac{1}{s}\Big)^{\frac12},\Big(\frac{1}{s}\Big)^{\frac32}\bigg),\quad\epsilon\in]0,\frac12[,
\end{equation*}
and
\begin{equation*}
|F^{(k)}(s)|\lesssim\min\bigg(\Big(\frac{1}{s}\Big)^{k},\Big(\frac{1}{s}\Big)^{k+\frac12},\Big(\frac{1}{s}\Big)^{k+\frac32}\bigg),\quad s\in]0,\infty[.
\end{equation*}
\item[{\bf(iv)}] The maps $s\mapsto s^{\alpha}F(s)$ and $s\mapsto s^{\beta}F'(s)$ are bounded for $0<\alpha\le\frac32$ and $1\le\beta\le\frac52$ respectively.
\end{itemize}
\end{prop}
Now, let 
\begin{equation}\label{G-kernel}
\mathcal{K}(r,z,\widetilde{r},\widetilde{z})=\frac{\sqrt{\widetilde{r}r}}{2\pi}F\Bigg(\frac{(r-\widetilde{r})^2+(z-\widetilde{z})^2}{\widetilde{r}r}\Bigg).
\end{equation}
Thus in view of \eqref{Psi}, $\Psi$ takes the form
\begin{equation*}
\Psi(r,z)=\int_{-\infty}^{\infty}\int_{0}^{\infty}\mathcal{K}(r,z,\widetilde{r},\widetilde{z})\omega_{\theta}(\widetilde{r},\widetilde{z})d\widetilde{r}d\widetilde{z},
\end{equation*}
with $\mathcal{K}$ can be seen as the kernel of the last integral representation. The last formula together with \eqref{v(r),v(z)} claim that there exists a genuine connection between the velocity and  its vorticity, namely, {\it axisymmatric Biot-Savart} law which reads as follows
\begin{equation}\label{axisym-BS}
v^{r}(r,z)=\int_{-\infty}^{\infty}\int_{0}^{\infty}\mathcal{K}_{r}(r,z,\widetilde{r},\widetilde{z})\omega_{\theta}(\widetilde{r},\widetilde{z})d\widetilde{r}d\widetilde{z},\quad v^z(r,z)=\int_{-\infty}^{\infty}\int_{0}^{\infty}\mathcal{K}_{z}(r,z,\widetilde{r},\widetilde{z})\omega_{\theta}(\widetilde{r},\widetilde{z})d\widetilde{r}d\widetilde{z}.
\end{equation} 
Here, with the notation $\xi^2=\frac{(r-\widetilde{r})^2+(z-\widetilde{z})^2}{\widetilde{r}r}$ we have 
\begin{equation}\label{K(r)}
\mathcal{K}_{r}(r,z,\widetilde{r},\widetilde{z})=-\frac{1}{\pi}\frac{z-\widetilde{z}}{r^{3/2}\widetilde{r}^{1/2}}F'(\xi^2)
\end{equation}
and
\begin{equation}\label{K(z)}
\mathcal{K}_{z}(r,z,\widetilde{r},\widetilde{z})=\frac{1}{\pi}\frac{r-\widetilde{r}}{r^{3/2}\widetilde{r}^{1/2}}F'(\xi^2)+\frac{1}{4\pi}\frac{\widetilde{r}^{1/2}}{r^{3/2}}\big(F(\xi^2)-2\xi^2F'(\xi^2)\big).
\end{equation}
A worthwhile property of the kernals $\mathcal{K}_{r}$ and $\mathcal{K}_{z}$ are given in the following result. For more details about the proof, see \cite{Gallay-Sverak}.
\begin{prop}\label{k(r),k(z)} Let $(r,z,\widetilde{r},\widetilde{z})\in \Omega\times\Omega$, then we have
\begin{equation}\label{k(r),k(z)}
|\mathcal{K}_{r}(r,z,\widetilde{r},\widetilde{z})|+|\mathcal{K}_{z}(r,z,\widetilde{r},\widetilde{z})|\le\frac{C}{\big((r-\widetilde{r})^2+(z-\widetilde{z})^2\big)^{1/2}}.
\end{equation}
\end{prop}
\hspace{0.7cm}Now, we state the first consequence of the above result, in particular the $L^p\rightarrow L^q$ between the velocity and its vorticity, specifically we establish.
\begin{prop}\label{p-q} Let $v$ be an axisymmetric velocity vector associated to the vorticity $\omega_{\theta}$ via the axisymmetric Biot-Savart law \eqref{axisym-BS}. Then the following assertions are hold.
\begin{enumerate}
\item[{\bf(i)}] Let $(p,q)\in(1,2)\times(2,\infty)$, with $p<q$ and $\frac{1}{p}-\frac{1}{q}=\frac{1}{2}$. For $\omega_{\theta}\in L^{p}(\Omega)$, then $v\in (L^q(\Omega))^2$ and 
\begin{equation}\label{estimate-v-omega}
\|v\|_{L^q(\Omega)}\le C\|\omega_{\theta}\|_{L^p(\Omega)}.
\end{equation}
\item[{\bf(ii)}] Let $(p,q)\in[1,2)\times(2,\infty]$, with $p<q$, and define $\sigma\in]0,1[$ by $\frac12=\frac{\sigma}{p}+\frac{1-\sigma}{q}$. Then for $\omega_{\theta}\in L^p(\Omega)\cap L^q(\Omega)$, we have $v\in( L^\infty(\Omega))^2$ and 
\begin{equation}\label{estimate-v-interpolation}
\|v\|_{L^\infty(\Omega)}\le C\|\omega_{\theta}\|^{\sigma}_{L^p(\Omega)}\|\omega_{\theta}\|^{1-\sigma}_{L^q(\Omega)}.
\end{equation}
\end{enumerate} 
\end{prop}
\begin{proof}
{\bf(i)} Combining \eqref{axisym-BS} and \eqref{k(r),k(z)}, we get
\begin{equation*}
|v^{r}(r,z)|\le C\int_{-\infty}^{\infty}\int_{0}^{\infty}\frac{|\omega_{\theta}(\widetilde{r},\widetilde{z})|}{\big((r-\widetilde{r})^2+(z-\widetilde{z})^2\big)^{1/2}}d\widetilde{r}d\widetilde{z},
\end{equation*}
and
\begin{equation*}
|v^{z}(r,z)|\le C\int_{-\infty}^{\infty}\int_{0}^{\infty}\frac{|\omega_{\theta}(\widetilde{r},\widetilde{z})|}{\big((r-\widetilde{r})^2+(z-\widetilde{z})^2\big)^{1/2}}d\widetilde{r}d\widetilde{z}.
\end{equation*}
The last two integrals of the right-hand side are be seen as a singular integral. So, by hypothesis $\frac1p -\frac{1}{q}=\frac12$, Hardy-Littlewood-Sobolev inequality, see e.g. [23, Theorem 6.1.3] yields the desired estimate.\\ 
{\bf(ii)} Let $R>0$, then in view of \eqref{k(r),k(z)} we have. 
\begin{equation*}
|v(r,z)|\lesssim\int_{\Omega_R}\frac{|\omega_{\theta}(r-\widetilde{r},z-\widetilde{z})|}{(\widetilde{r}^2+\widetilde{z}^2)^{\frac12}}d\widetilde{r}d\widetilde{z}+\int_{\Omega\setminus\Omega_1R}\frac{|\omega_{\theta}(r-\widetilde{r},z-\widetilde{z})|}{(\widetilde{r}^2+\widetilde{z}^2)^{\frac12}}d\widetilde{r}d\widetilde{z},
\end{equation*}
where $\Omega_R=\lbrace (r,z)\in\Omega :0<r\leq R,  -R\leq z\leq R\rbrace$. Thus, H\"older's inequality implies
\begin{equation*}
|v(r,z)|\lesssim\|\omega_{\theta}\|_{L^q(\Omega)}R^{1-\frac2q}+\|\omega_{\theta}\|_{L^p(\Omega)}\frac{1}{R^{\frac2p-1}}.
\end{equation*}
It is enough to take $R=\big(\|\omega_{\theta}\|_{L^p(\Omega)}/\|\omega_{\theta}\|_{L^q(\Omega)}\big)^{\ell}$, with $\ell=\frac{\sigma}{1-2/q}=\frac{1-\sigma}{2/p -1}$. Then by easy computations achieve the estimate.
\end{proof}
\hspace{0.7cm}In the axisymmetric case the weighted estimates practice a decisive role to bound some quantities like $r^{\alpha}v$ in Lebesgue spaces for some $\alpha$. Now, we state some of them which their proofs can be found in \cite{Feng-Sverak,Gallay-Sverak}.
\begin{prop}\label{Prop-weighted-estimate} Let $\alpha,\beta\in[0,2]$ be such that $\beta-\alpha\in[0,1)$, and assume that $p,q\in(1,\infty)$ satisfying
\begin{equation*}
\frac{1}{p}-\frac{1}{q}=\frac{1+\alpha-\beta}{2}.
\end{equation*}
Assume that $r^{\beta}\omega_{\theta}\in L^{p}(\Omega)$, then $r^{\alpha}v\in (L^{q}(\Omega))^2$ and the following bound holds true.
\begin{equation}\label{weighted-estimate}
\|r^{\alpha}v\|_{L^q(\Omega)}\le C\|r^{\beta}\omega_{\theta}\|_{L^p(\Omega)}.
\end{equation}
\end{prop}

\subsection{Characterizations of semigroups associated with the linearized equation} We focus on studying the linearized boundary initial value problem associated to the system \eqref{VD} and we state some properties of their semipgroups. Specifically, we consider
\begin{equation}\label{lin-prob}
\left\{\begin{array}{ll}
\partial_{t}\omega_{\theta}-\Big(\Delta-\frac{1}{r^2}\Big)\omega_{\theta}=0, & \\
\partial_{t}\rho-\Delta\rho=0, &\\
(\omega_{\theta},\rho)_{|t=0}=(\omega_0,\rho_0)
\end{array}
\right.
\end{equation}
in the product space $\Omega\times\RR^3$, with $\Omega=\{(r,z)\in\RR^2: r>0\}$ is the half-space by prescribing the homogeneous Dirichlet conditions at the boundary $r=0$ for $\omega_{\theta}$ variable. For $(\omega_0,\rho_0)\in L^1(\Omega)\times L^1(\RR^3)$, the solution of \eqref{lin-prob} is given explicitely by 
\begin{equation*}
\left\{\begin{array}{ll}
\omega_{\theta}(t)=\Ss_1(t)\omega_{0},\\ \rho(t)=\Ss_2(t)\rho_0,
\end{array}
\right.
\end{equation*}
where $(\Ss_1(t))_{t\ge0}$ and $(\Ss_2(t))_{t\ge 0}$ being respectively the semigroups or evolution operators associated to the dissipative operators $(\Delta -\frac{1}{r^2})$ and $\Delta$. 

\hspace{0.7cm}Such  are characterized by the following explicit formulae, namely we have.
\begin{prop}\label{S1,S2} The family $(\Ss_1(t),\Ss_2(t))_{t\ge0}$ associated to \eqref{lin-prob} is expressed by the following
\begin{equation}\label{exp-exp}
\left\{\begin{array}{ll}
(\Ss_1(t)\omega_{0})(r,z)=\frac{1}{4\pi t}\int_{\Omega}\frac{\widetilde{r}^{1/2}}{r^{1/2}}\mathscr{N}_1\Big(\frac{t}{r\widetilde{r}}\Big)e^{-\frac{(r-\widetilde{r})^2+(z-\widetilde{z})^2}{4t}}\omega_{0}(\widetilde{r},\widetilde{z})d\widetilde{r}d\widetilde{z}, &\\
(\Ss_2(t)\rho_{0})(r,z)=\frac{1}{4\pi t}\int_{\Omega}\frac{\widetilde{r}^{1/2}}{r^{1/2}}\mathscr{N}_2\Big(\frac{t}{r\widetilde{r}}\Big)e^{-\frac{(r-\widetilde{r})^2+(z-\widetilde{z})^2}{4t}}\rho_{0}(\widetilde{r},\widetilde{z})d\widetilde{r}d\widetilde{z},
\end{array}
\right.
\end{equation}
where the functions  $(0,+\infty)\ni t\mapsto \mathscr{N}_1(t), \mathscr{N}_2(t)\in\RR$ are defined by 
\begin{equation}\label{N1-N2-function}
\left\{\begin{array}{ll}
\mathscr{N}_1(t)=\frac{1}{\sqrt{\pi t}}\int_{-\pi/2}^{\pi/2}e^{-\frac{\sin^2\alpha}{t}}\cos(2\alpha)d\alpha, &\\
\mathscr{N}_2(t)=\frac{1}{\sqrt{\pi t}}\int_{-\pi/2}^{\pi/2}e^{-\frac{\sin^2\alpha}{t}}d\alpha.
\end{array}
\right.
\end{equation}
\end{prop}
\begin{proof} We assume that $(\omega_{\theta},\rho)$ solving \eqref{lin-prob}, then a straigthforward computations claim that $(\omega,\rho)$, with $\omega=\omega_{\theta}\vec{e}_{\theta}$ satisfying the usual heat equation  $\partial_t\omega-\Delta\omega=0$ and $\partial_t\rho-\Delta\rho=0$ in $\RR^3$ with initial data $(\omega(0,\cdot),\rho(0,\cdot))$. Therefore, for every $t>0$ we have 
\begin{equation}\label{us-heat-eq}
\left\{\begin{array}{ll}
\omega(t,x)=\frac{1}{(4\pi t)^{3/2}}\int_{\RR^3}e^{-\frac{\vert x-\widetilde{x}\vert^2}{4t}}\omega(0,\widetilde{x})d\widetilde{x}
, &\\
\rho(t,x)=\frac{1}{(4\pi t)^{3/2}}\int_{\RR^3}e^{-\frac{\vert x-\widetilde{x}\vert^2}{4t}}\rho(0,\widetilde{x})d\widetilde{x}.
\end{array}
\right.
\end{equation}
We will develop each term in the cylindrical basis $(\vec e_r,\vec e_\theta,\vec e_z)$ by writing $x=(r\cos\theta, r\sin\theta,z)$ and $\widetilde{x}=(\widetilde{r}\cos\widetilde{\theta}, \widetilde{r}\sin\widetilde{\theta},\widetilde{z})$, hence the first  equation of \eqref{us-heat-eq} takes the form
\begin{eqnarray}\label{Formula1}
\\
\nonumber\omega_{\theta}(t,r,z)\left(\begin{array}{c}
-\sin\theta \\
\cos\theta \\
0
\end{array}\right)
&=&\frac{1}{(4\pi t)^{3/2}}\int_{0}^{\infty}\int_{\RR}\int_{-\pi}^{\pi}e^{-\frac{\vert x-\widetilde{x}\vert^2}{4t}}\omega_{0}(\widetilde{r},\widetilde{z})
\left(\begin{array}{c}
-\sin\widetilde{\theta} \\
\cos\widetilde{\theta} \\
0
\end{array}\right)
\widetilde{r}d\widetilde{\theta}d\widetilde{z}d\widetilde{r}\\
\nonumber &=& \textnormal{I}_1.
\end{eqnarray}
Since, $\vert x-\widetilde{x}\vert^2=(r-\widetilde{r})^{2}+(z-\widetilde{z})^{2}+4r\widetilde{r}\sin^{2}\big(\frac{\theta-\widetilde{\theta}}{2}\big)$, thus we have
\begin{equation}\label{I}
\textnormal{I}_1=\frac{1}{(4\pi t)}\int_{0}^{\infty}\int_{\RR}\Bigg(\frac{1}{(4\pi t)^{1/2}}\int_{-\pi}^{\pi}e^{\frac{-r\widetilde{r}\sin^{2}\frac{\theta-\widetilde{\theta}}{2}}{t}}\left(\begin{array}{c}
-\sin\widetilde{\theta} \\
\cos\widetilde{\theta} \\
0
\end{array}\right)\widetilde{r}d\widetilde{\theta}\Bigg)e^{-\frac{(r-\widetilde{r})^{2}+(z-\widetilde{z})^{2}}{4t}}\omega_{0}(\widetilde{r},\widetilde{z})d\widetilde{z}d\widetilde{r}.
\end{equation}
To treat I$_1$, we set $\alpha=\frac{\theta-\widetilde{\theta}}{2}$ then we have
\begin{eqnarray*}
\frac{1}{\sqrt{4\pi t}}\int_{-\pi}^{\pi}e^{\frac{-r\widetilde{r}\sin^{2}\frac{\theta-\widetilde{\theta}}{2}}{t}}(-\sin\widetilde{\theta})\widetilde{r}d\widetilde{\theta}&=&-\frac{1}{\sqrt{\pi t}}\int_{\theta/2-\pi/2}^{\theta/2+\pi/2}e^{\frac{-r\widetilde{r}\sin^{2}\alpha}{t}}\big(\sin\theta\cos 2\alpha-\cos\theta\sin 2\alpha\big)\widetilde{r}d\alpha\\
&=&-\frac{1}{\sqrt{\pi t}}\int_{-\pi/2}^{+\pi/2}e^{\frac{-r\widetilde{r}\sin^{2}\alpha}{t}}\big(\sin\theta\cos 2\alpha\big) \widetilde{r}d\alpha\\
&&-\frac{1}{\sqrt{\pi t}}\int_{-\pi/2}^{+\pi/2}e^{\frac{-r\widetilde{r}\sin^{2}\alpha}{t}}\big(\cos\theta\sin 2\alpha\big) \widetilde{r}d\alpha.
\end{eqnarray*}
For $t\in(0,\infty)$, define
\begin{equation*}
\mathscr{N}_1(t)=\frac{1}{\sqrt{\pi t}}\int_{-\pi/2}^{\pi/2}e^{-\frac{\sin^2\alpha}{t}}\cos(2\alpha)d\alpha.
\end{equation*}
Then the last estimate becomes
\begin{equation*}
\frac{1}{(4\pi t)^{\frac12}}\int_{-\pi}^{\pi}e^{\frac{-r\widetilde{r}\sin^{2}\frac{\theta-\widetilde{\theta}}{2}}{t}}(-\sin\widetilde{\theta})\widetilde{r}d\widetilde{\theta}=\frac{\widetilde{r}^{1/2}}{r^{1/2}}\mathscr{N}_1\Big(\frac{t}{r\widetilde{r}}\Big)(-\sin\theta).
\end{equation*}
Similarly,
\begin{equation*}
\frac{1}{(4\pi t)^{1/2}}\int_{-\pi}^{\pi}e^{\frac{-r\widetilde{r}\sin^{2}\frac{\theta-\widetilde{\theta}}{2}}{t}}\cos\widetilde{\theta}\widetilde{r}d\widetilde{\theta}=\frac{\widetilde{r}^{1/2}}{r^{1/2}}\mathscr{N}_1\Big(\frac{t}{r\widetilde r}\Big)\cos\theta.
\end{equation*}
Combining the last two estimates and plug them in I$_1$ we reach the desired estimate. 

\hspace{0.5cm}For the second equation in \eqref{us-heat-eq} we express the density formula in $(\vec e_r,\vec e_\theta,\vec e_z)$ basis 
\begin{equation}\label{us-heat-eq1}
\rho(t,r,z)=\frac{1}{4\pi t}\int_{\Omega}\bigg(\frac{1}{2\sqrt{\pi t}}\int_{-\pi}^{\pi}e^{-\frac{r\widetilde{r}\sin^2(\frac{\theta-\widetilde{\theta}}{2})}{t}}\widetilde{r}d\widetilde{\theta}\bigg)e^{-\frac{(r-\widetilde{r})^2+(z-\widetilde{z})^2}{4t}}\rho(0,\widetilde{r},\widetilde{z})d\widetilde{r}d\widetilde{z}.
\end{equation}
Setting
\begin{equation*}
\textnormal{I}_2=\frac{1}{2\sqrt{\pi t}}\int_{-\pi}^{\pi}e^{-\frac{r\widetilde{r}\sin^2\big(\frac{\theta-\widetilde{\theta}}{2}\big)}{t}}\widetilde{r}d\widetilde{\theta}.
\end{equation*}
The same variable  $\alpha=\frac{\theta-\widetilde{\theta}}{2}$ allows us to write
\begin{equation*}\label{us-heat-eq2}
\textnormal{I}_2 =\frac{1}{\sqrt{\pi t}}\int_{-\pi/2}^{\pi/2}e^{-\frac{r\widetilde{r}\sin^2\alpha}{t}}\widetilde{r}d\alpha=\sqrt{\frac{\widetilde{r}}{r}}\mathscr{N}_2\Big(\frac{t}{r\widetilde{r}}\Big),
\end{equation*}
with $\mathscr{N}_2$ is defined for $t>0$ by
\begin{equation*}
\mathscr{N}_2(t)=\frac{1}{\sqrt{\pi t}}\int_{-\pi/2}^{\pi/2}e^{-\frac{\sin^2\alpha}{t}}d\alpha.
\end{equation*}
Plug I$_2$ in \eqref{us-heat-eq1}, we get the result. This ends the proof of the Proposition.
\end{proof}   
\hspace{0.7cm}The following Proposition provides some  asymptotic behavior of the functions $\mathscr{N}_1$ and $\mathscr{N}_2$ near $0$ and $\infty$, which will be fundamental in the sequel. 
\begin{prop} Let $\mathscr{N}_1, \mathscr{N}_2:(0,\infty)\rightarrow\RR$ be the functions defined in \eqref{N1-N2-function}. Then the following statements are hold.   
\begin{enumerate}
\item[{\bf(i)}] $\mathscr{N}_1(t)=1-\frac{3t}{4}+O(t^2)$ and  $\mathscr{N}'_1(t)=-\frac{3}{4}+O(t)$  when  $t\uparrow0$;
\item[{\bf(ii)}] $\mathscr{N}_1(t)=\frac{\pi^{1/2}}{4t^{3/2}}+O\Big(\frac{1}{t^{5/2}}\Big)$ and $\mathscr{N}_1'(t)=-\frac{3\pi^{1/2}}{8t^{5/2}}+O\Big(\frac{1}{t^{7/2}}\Big)$ when $t\uparrow\infty$;
\item[{\bf(iii)}] $\mathscr{N}_2(t)=1-\frac{t}{4}+O(t^2)$ and  $\mathscr{N}'_2(t)=-\frac{1}{4}+O(t)$  when  $t\uparrow 0$; 
\item[{\bf(iv)}] $\mathscr{N}_2(t)=\frac{\pi^{1/2}}{t^{1/2}}-\frac{\pi^{5/2}}{12t^{3/2}}+O\Big(\frac{1}{t^{5/2}}\Big)$ and $\mathscr{N}'_2(t)=-\frac{\pi^{1/2}}{12t^{3/2}}-\frac{\pi^{5/2}}{8t^{5/2}}+O\Big(\frac{1}{t^{7/2}}\Big)$ when $t\uparrow\infty$.
\end{enumerate}
\end{prop}  
\begin{proof}
{\bf(i)} Substituting $\zeta=\frac{\sin\alpha}{\sqrt{t}}$ in $\mathscr{N}_1$, we shall have
\begin{eqnarray*}
\mathscr{N}_1(\zeta)&=&\frac{1}{\sqrt{\pi}}\int_{-\frac{1}{\sqrt{t}}}^{\frac{1}{\sqrt{t}}}e^{-\zeta^2}\frac{1-2t \zeta^2}{\sqrt{1-t\zeta^2}}d\zeta\\
&=&\frac{2}{\sqrt{\pi}}\int_{0}^{\frac{1}{\sqrt{t}}}e^{-\zeta^2}\frac{1-2t \zeta^2}{\sqrt{1-t\zeta^2}}d\zeta\\
&=& \frac{2}{\sqrt{\pi}}\bigg(\int_{0}^{\frac{1}{2\sqrt{t}}}e^{-\zeta^2}\frac{1-2t \zeta^2}{\sqrt{1-t\zeta^2}}d\zeta+\int_{\frac{1}{2\sqrt{t}}}^{\frac{1}{\sqrt{t}}}e^{-\zeta^2}\frac{1-2t \zeta^2}{\sqrt{1-t\zeta^2}}d\zeta\bigg)\\
&=& \textnormal{II}_1+\textnormal{II}_2.
\end{eqnarray*}
Note that $\lim_{t\uparrow0}\textnormal{II}_2=0 $, so the behavior of $\mathscr{N}_1 $ near $0$ comes from $\textnormal{II}_1$. Hence, let us deal with $\textnormal{II}_1$, we insert the Taylor expansion of the function $ \zeta\rightarrow\frac{1}{\sqrt{1-t\zeta^2}}$ in the integral of $\textnormal{II}_1$ to obtain
\begin{equation*}
\textnormal{II}_1=\frac{2}{\sqrt{\pi}}\int_{0}^{\frac{1}{2\sqrt{t}}}e^{-\zeta^2}(1-\frac{3}{2}t\zeta^2-t^2\zeta^4)d\zeta +O(t^3).
\end{equation*}
It is straightforward to show that 
$$
\int_0^\infty e^{-\zeta^2}d\zeta=\frac{\sqrt{\pi}}{2},\quad\int_0^\infty \zeta^2e^{-\zeta^2}d\zeta=\frac{\sqrt{\pi}}{4},\quad\int_0^\infty\zeta^4e^{-\zeta^2}d\zeta=-3\frac{\sqrt{\pi}}{8}.
$$

Consequently, $\lim_{t\uparrow0}\mbox{II}_1=1$. Combining all the previous quantities, we find the asymptotic behavior of $\textnormal{II}_1$ near $0$, that is,
$$
\textnormal{II}_1=1-\frac{3}{4}t+O(t^2).
$$ 
By derivation of $\textnormal{II}_1$, we find the behavior of $\mathscr{N}'_1 $. \\
{\bf (ii)} The Mac Laurin's expansion of the function $\alpha\mapsto e^{-\frac{\sin^2\alpha}{t}}$ at $0$ is given by
\begin{equation*}
e^{-\frac{\sin^2\alpha}{t}}=1-\frac{\alpha^2}{t}+O\Big(\frac{1}{t^2}\Big).
\end{equation*}
Thus we get
\begin{equation*}
\mathscr{N}_1(t)=\frac{1}{\sqrt{\pi t}}\int_{-\pi/2}^{\pi/2}\big(1-\frac{\alpha^2}{t}\big)\cos 2\alpha d\alpha +O\Big(\frac{1}{t^{\frac52}}\Big). 
\end{equation*}
After an easy computations we achieve the estimate.\\
{\bf(iii)} To prove this assertion, setting $y=\frac{\sin\alpha}{\sqrt{t}}$ in $\mathscr{N}_2$ and we split the integral into two parts, one has
\begin{equation*}
\mathscr{N}_2(t) = \frac{2}{\sqrt{\pi}}\bigg(\int_{0}^{\frac{1}{2\sqrt{t}}}\frac{e^{-y^2}}{\sqrt{1-ty^2}}dy+\int_{\frac{1}{2\sqrt{t}}}^{\frac{1}{\sqrt{t}}}\frac{e^{-y^2}}{\sqrt{1-ty^2}}dy\bigg).
\end{equation*}
We follow the same steps as  $\mathscr{N}_1 $.  For the second integral in right-hand side, we have
\begin{eqnarray*}
\frac{2}{\sqrt{\pi}}\int_{\frac{1}{2\sqrt{t}}}^{\frac{1}{\sqrt{t}}}\frac{e^{-y^2}}{\sqrt{1-ty^2}}dy&=&\frac{2}{\sqrt{\pi}}\int_{\frac{1}{2\sqrt{t}}}^{\frac{1}{\sqrt{t}}}\frac{e^{-y^2}}{(\sqrt{1-\sqrt{t}y})(\sqrt{1+\sqrt{t}y})}dy\\
&\leq& Ce^{-\frac{1}{4t}}\int_{\frac{1}{2\sqrt{t}}}^{\frac{1}{\sqrt{t}}}\frac{1}{\sqrt{1-\sqrt{t}y}}dy.
\end{eqnarray*}
Let us observe that the last estimate goes to $0$ as $t\uparrow0$, so the asymptotic behavior of $\mathscr{N}_2$ near $0$ comes only from the first integral. To be precise, it is clear that $t\mapsto \frac{1}{\sqrt{1-ty^2}}$ is bounded function whenever $0<y<\frac{1}{2\sqrt{t}}$ and 
\begin{equation*}
\lim_{t\uparrow0}\frac{2}{\sqrt{\pi}}\int_{0}^{\frac{1}{2\sqrt{t}}}e^{-y^2}dy\approx 1.  
\end{equation*}
Thus, the expansion of the function $x\mapsto(1-x)^{-\frac12}$ for $x=ty^2$ enuble us to write
\begin{eqnarray*}
\mathscr{N}_2(t)& =& \frac{2}{\sqrt{\pi}}\int_{0}^{\frac{1}{2\sqrt{t}}}e^{-y^2}\Big(1+\frac{ty^2}{2}\Big)dy+O(t^2)\\
&=&1-\frac{t}{4}+O(t^2).
\end{eqnarray*}
{\bf(iv)} Using the fact $\sin\alpha\simeq \alpha$ near $0$, then we get 
\begin{equation*}
\mathscr{N}_2(t)=\frac{1}{\sqrt{\pi t}}\int_{-\frac{\pi}{2}}^{\frac{\pi}{2}}e^{-\frac{\alpha^2}{t}}d\alpha.
\end{equation*}
We set $y=\frac{\alpha}{\sqrt{t}}$, clearly that $y\uparrow 0$ as $ t\uparrow\infty$ and the power expansion of the function $e^{y}$ near $0$ yields the asymptoyic expansion, whereas $\mathscr{N}'_2 $ is a direct derivative of $\mathscr{N}_2 $ expansion.\\  
\end{proof}
\hspace{0.7cm}Some consequences of the previous Proposition are listed in the following remark. 
\begin{Rema}\label{Rem-Weight}
\begin{enumerate}
\item[{\bf(i)}] It should be noted that the functions $t\mapsto\mathscr{N}_1(t)$ and $t\mapsto\mathscr{N}_2(t)$ are decreasing over $]0,\infty[$, but the proof seems very hard.   
\item[{\bf(ii)}] The functions $t\mapsto t^{\alpha}\mathscr{N}_1(t), t\mapsto t^{\alpha}\mathscr{N}_2(t)$ and $t\mapsto t^{\beta}\mathscr{N}'_1(t),t\mapsto t^{\beta}\mathscr{N}'_2(t)$ are bounded for $0\leq\alpha\leq \frac{1}{2}$ and $0\leq\beta\leq \frac{3}{2}$.
\end{enumerate}
\end{Rema}
\hspace{0.5cm} Other nice properties of $(\Ss_i(t))_{t\ge0}$, with $i=1,2$, in particular the estimate $L^p\rightarrow L^q$ are given in the following result. 
\begin{prop}\label{P,S1,S2} The family $((\Ss_1(t),\Ss_2(t))_{t\ge0}$ associated to \eqref{lin-prob} is a strongly continuous semigroup of bounded linear operators in $L^p(\Omega)\times L^p(\Omega) $ for any $p\in[1,\infty]$. Furtheremore, for $1\le p\le q\le\infty$ the following assertions are hold.
\begin{enumerate}
\item[{\bf(i)}] For $(\omega_0,\rho_0)\in L^p(\Omega)\times L^p(\Omega)$, we have for every $t>0$ 
\begin{equation}\label{S(t)}
\|(\Ss_1(t)\omega_0,\Ss_2(t)\rho_0)\|_{L^q(\Omega)\times L^q(\Omega)}\le\frac{C}{t^{\frac{1}{p}-\frac{1}{q}}}\|(\omega_0,\rho_0)\|_{L^p(\Omega)\times L^p(\Omega)}.
\end{equation}
\item[{\bf(ii)}] For $f=(f^r,f^z)\in L^p(\Omega)\times L^p(\Omega)$, we have for every $t>0$
\begin{equation}\label{S(t)div}
\|\Ss_1(t)\Div_{\star}f\|_{L^q(\Omega)}\le\frac{C}{t^{\frac{1}{2}+\frac{1}{p}-\frac{1}{q}}}\|f\|_{L^p(\Omega)}.
\end{equation}
\item[{\bf(iii)}] For $f=(f^r,f^z)\in L^p(\Omega)\times L^p(\Omega)$, we have every $t>0$
\begin{equation}\label{S2(t)div}
\|\Ss_2(t)\Div f\|_{L^q(\Omega)}\le\frac{C}{t^{\frac{1}{2}+\frac{1}{p}-\frac{1}{q}}}\|f\|_{L^p(\Omega)}.
\end{equation}
\end{enumerate}
Here, $\Div_{\star}f=\partial_r f^r+\partial_z f^z$ (resp.  $\Div f=\partial_r f^r+\partial_z f^z+\frac{f^r}{r}$) stands the divergence operator over $\RR^2$ (resp. the divergence operator over $\RR^3$ in the axisymmetric case).
\end{prop} 
\begin{proof} {\bf (i)} We follow the proof of \cite{Gallay-Sverak} with minor modifications,  for this aim let $(r,z),(\widetilde{r},\widetilde{z})\in\Omega$, we will prove the following worth while estimates
\begin{equation}\label{Est:HZ1}
\left\{\begin{array}{ll}
\frac{1}{4\pi t}\frac{\widetilde{r}^{1/2}}{r^{1/2}}\mathscr{N}_1\Big(\frac{t}{r\widetilde{r}}\Big)e^{-\frac{(r-\widetilde{r})^2+(z-\widetilde{z})^2}{4t}}\le\frac{C}{t}e^{-\frac{(r-\widetilde{r})^2+(z-\widetilde{z})^2}{5t}},&\\
\frac{1}{4\pi t}\frac{\widetilde{r}^{1/2}}{r^{1/2}}\mathscr{N}_2\Big(\frac{t}{r\widetilde{r}}\Big)e^{-\frac{(r-\widetilde{r})^2+(z-\widetilde{z})^2}{4t}}\le\frac{C}{t}e^{-\frac{(r-\widetilde{r})^2+(z-\widetilde{z})^2}{5t}}.
\end{array}\right.
\end{equation}
We distinguish two cases $\widetilde{r}\le 2r$ and $\widetilde{r}>2r$.\\
{$\bullet$ $\widetilde{r}\le 2r$}. Employing the fact $t\mapsto (t^{\alpha}\mathscr{N}_1(t),t^{\alpha}\mathscr{N}_2(t))$ is bounded for $\alpha\in[0,\frac12]$, see, {\bf(ii)}-Remark \ref{Rem-Weight} and $t\mapsto e^{-t}$ is decreasing, we get the result.\\
{$\bullet$ $\widetilde{r}>2r$}. The remark $\widetilde{r}\leq 2\big((r-\widetilde{r})^2+(z-\widetilde{z})^2\big)^{\frac12}$, a new use of $t\mapsto (t^{\alpha}\mathscr{N}_1(t),t^{\alpha}\mathscr{N}_2(t))$ is bounded for $\alpha\in[0,\frac12]$  and $te^{-\frac{t^2}{4}}\le Ce^{-\frac{t^2}{5}}$ for $t\ge0$ leading to 
\begin{eqnarray*}
\frac{1}{4\pi t}\frac{\widetilde{r}^{1/2}}{r^{1/2}}\mathscr{N}_i\Big(\frac{t}{r\widetilde{r}}\Big)e^{-\frac{(r-\widetilde{r})^2+(z-\widetilde{z})^2}{4t}}&\le &\frac{C}{t}\bigg(\frac{(r-\widetilde{r})^2+(z-\widetilde{z})^2}{4t}\bigg)^{\frac12}e^{-\frac{(r-\widetilde{r})^2+(z-\widetilde{z})^2}{4t}}\\
&\le & \frac{C}{t}e^{-\frac{(r-\widetilde{r})^2+(z-\widetilde{z})^2}{5t}},\quad i\in\{1,2\}.
\end{eqnarray*}

Next, from \eqref{Est:HZ1} and the last estimate we write
\begin{eqnarray*}
|\Ss_1(t)\omega_0|+|\Ss_2(t)\rho_0| &\le & \frac{1}{4\pi t}\int_{\Omega}\bigg|\frac{\widetilde{r}^{1/2}}{r^{1/2}}\mathscr{N}_1\Big(\frac{t}{r\widetilde{r}}\Big)e^{-\frac{(r-\widetilde{r})^2+(z-\widetilde{z})^2}{4t}}\omega_{0}(\widetilde{r},\widetilde{z})\bigg|d\widetilde{r}d\widetilde{z}\\
&&+ \frac{1}{4\pi t}\int_{\Omega}\bigg|\frac{\widetilde{r}^{1/2}}{r^{1/2}}\mathscr{N}_2\Big(\frac{t}{r\widetilde{r}}\Big)e^{-\frac{(r-\widetilde{r})^2+(z-\widetilde{z})^2}{4t}}\rho_{0}(\widetilde{r},\widetilde{z})\bigg|d\widetilde{r}d\widetilde{z}\\
 &\leq&\frac{C}{t}\int_{\Omega}e^{-\frac{(r-\widetilde{r})^2+(z-\widetilde{z})^2}{5t}}\big(|\omega_{0}(\widetilde{r},\widetilde{z})|+|\rho_0(\widetilde{r},\widetilde{z})|\big)d\widetilde{r}d\widetilde{z}.
\end{eqnarray*} 
The last line can be seen as a convolution product, then Young's inequality gives the desired estimate.\\      
{\bf(ii)} By definition for every $(r,z)\in\Omega$, we have
\begin{eqnarray*}
\big(\Ss_1(t)\Div_{\star}f\big)(r,z)&=&\frac{1}{4\pi t}\int_{\Omega}\frac{\widetilde{r}^{1/2}}{r^{1/2}}\mathscr{N}_1\Big(\frac{t}{r\widetilde{r}}\Big)e^{-\frac{(r-\widetilde{r})^2+(z-\widetilde{z})^2}{4t}}(\partial_{\widetilde{r}}f^{r}(\widetilde{r},\widetilde{z})+\partial_{\widetilde{z}}f^{z}(\widetilde{r},\widetilde{z}))d\widetilde{r}d\widetilde{z}\\
&=&\frac{1}{4\pi t}\int_{\Omega}\frac{\widetilde{r}^{1/2}}{r^{1/2}}\mathscr{N}_1\Big(\frac{t}{r\widetilde{r}}\Big)e^{-\frac{(r-\widetilde{r})^2+(z-\widetilde{z})^2}{4t}}\partial_{\widetilde{r}}f^{r}(\widetilde{r},\widetilde{z})d\widetilde{r}d\widetilde{z}\\
&+&\frac{1}{4\pi t}\int_{\Omega}\frac{\widetilde{r}^{1/2}}{r^{1/2}}\mathscr{N}_1\Big(\frac{t}{r\widetilde{r}}\Big)e^{-\frac{(r-\widetilde{r})^2+(z-\widetilde{z})^2}{4t}}\partial_{\widetilde{z}}f^{z}(\widetilde{r},\widetilde{z})d\widetilde{r}d\widetilde{z}\\
&=&\textnormal{II}_{1}+\textnormal{II}_{2}.
\end{eqnarray*}
After an integration by parts, it happens
\begin{equation*}
\textnormal{II}_1=\frac{1}{4\pi t}\int_{\Omega}\frac{\widetilde{r}^{1/2}}{r^{1/2}}\bigg(\frac{t}{r\widetilde{r}^2}\mathscr{N}_1'\Big(\frac{t}{r\widetilde{r}}\Big)-\Big(\frac{1}{2\widetilde{r}}+\Big(\frac{r-\widetilde{r}}{2t}\Big)\Big)\mathscr{N}_1\Big(\frac{t}{r\widetilde{r}}\Big)\bigg)e^{-\frac{(r-\widetilde{r})^2+(z-\widetilde{z})^2}{4t}}f^{r}(\widetilde{r},\widetilde{z})d\widetilde{r}d\widetilde{z},
\end{equation*}
and
\begin{equation*}
\textnormal{II}_2=-\frac{1}{4\pi t}\int_{\Omega}\frac{\widetilde{r}^{1/2}}{r^{1/2}}\Big(\frac{z-\widetilde{z}}{2t}\Big)\mathscr{N}_1\Big(\frac{t}{r\widetilde{r}}\Big)e^{-\frac{(r-\widetilde{r})^2+(z-\widetilde{z})^2}{4t}}f^{z}(\widetilde{r},\widetilde{z})d\widetilde{r}d\widetilde{z}.
\end{equation*}
We proceed by the same manner as above. The fact that the functions $\mathscr{N}_1, \mathscr{N}'_1$ and $t\mapsto t^{\alpha}\mathscr{N}_1(t),t\mapsto t^{\alpha}\mathscr{N}'_1(t)$ are bounded, see Remark \ref{Rem-Weight}, one finds
\begin{equation*}\label{Ar}
|\textnormal{II}_1|\leq\frac{C}{t^\frac{3}{2}}\int_{\Omega}e^{-\frac{(r-\widetilde{r})^2+(z-\widetilde{z})^2}{5t}}|f^{r}(\widetilde{r},\widetilde{z})|d\widetilde{r}d\widetilde{z},
\end{equation*}
and
\begin{equation*}\label{Az}
|\textnormal{II}_2|\leq\frac{C}{t^\frac{3}{2}}\int_{\Omega}e^{-\frac{(r-\widetilde{r})^2+(z-\widetilde{z})^2}{5t}}|f^{z}(\widetilde{r},\widetilde{z})|d\widetilde{r}d\widetilde{z}.
\end{equation*}
Together with Young's inequality, we obtain \eqref{S(t)div}.\\ 
{\bf(iii)} Let $(r,z)\in\Omega$, then we have
\begin{eqnarray}\label{Est:01-1}
\nonumber\Ss_2(t)\Div f(r,z) &=& \frac{1}{4\pi t}\int_{\Omega}\frac{\widetilde{r}^{1/2}}{r^{1/2}}\mathscr{N}_2\Big(\frac{t}{r\widetilde{r}}\Big)e^{-\frac{(r-\widetilde{r})^2+(z-\widetilde{z})^2}{4t}}\Big(\partial_{\widetilde{r}}f^{r}(\widetilde{r},\widetilde{z})+\partial_{\widetilde{r}}f^{z}(\widetilde{r},\widetilde{z})+\frac{1}{\widetilde{r}}f^{r}(\widetilde{r},\widetilde{z})\Big)d\widetilde{r}d\widetilde{z}\\
\nonumber &=& \frac{1}{4\pi t}\int_{\Omega}\frac{\widetilde{r}^{1/2}}{r^{1/2}}\mathscr{N}_2\Big(\frac{t}{r\widetilde{r}}\Big)e^{-\frac{(r-\widetilde{r})^2+(z-\widetilde{z})^2}{4t}}\partial_{\widetilde{r}}f^{r}(\widetilde{r},\widetilde{z})\widetilde{r}d\widetilde{z}\\
\nonumber &&+\frac{1}{4\pi t}\int_{\Omega}\frac{\widetilde{r}^{1/2}}{r^{1/2}}\mathscr{N}_2\Big(\frac{t}{r\widetilde{r}}\Big)e^{-\frac{(r-\widetilde{r})^2+(z-\widetilde{z})^2}{4t}}\partial_{\widetilde{z}}f^{z}(\widetilde{r},\widetilde{z})\widetilde{r}d\widetilde{z}\\
 \nonumber &&+\frac{1}{4\pi t}\int_{\Omega}\frac{\widetilde{r}^{1/2}}{r^{1/2}}\mathscr{N}_2\Big(\frac{t}{r\widetilde{r}}\Big)e^{-\frac{(r-\widetilde{r})^2+(z-\widetilde{z})^2}{4t}}\frac{1}{\widetilde{r}}f^{r}(\widetilde{r},\widetilde{z})d\widetilde{r}d\widetilde{z}\\
 &=&\textnormal{III}_{3}+\textnormal{III}_{4}+\textnormal{III}_{5}.
\end{eqnarray}
The two terms \textnormal{III}$_3$ and \textnormal{III}$_4$ ensue by the same argument as in {\bf(ii)}. It remains to treat the term \textnormal{III}$_5$ in the following way
\begin{eqnarray}\label{Est:02-2}
\nonumber\textnormal{III}_{5}&=&\frac{1}{4\pi t}\int_{\Omega}\frac{\widetilde{r}^{1/2}}{r^{1/2}}\mathscr{N}_2\Big(\frac{t}{r\widetilde{r}}\Big)e^{-\frac{(r-\widetilde{r})^2+(z-\widetilde{z})^2}{4t}}\frac{1}{\widetilde{r}}f^{r}(\widetilde{r},\widetilde{z})d\widetilde{r}d\widetilde{z}\\
 &=&\frac{1}{4\pi t}\int_{\Omega}\frac{1}{(r\widetilde{r})^{1/2}}\mathscr{N}_2\Big(\frac{t}{r\widetilde{r}}\Big)e^{-\frac{(r-\widetilde{r})^2+(z-\widetilde{z})^2}{4t}}f^{r}(\widetilde{r},\widetilde{z})d\widetilde{r}d\widetilde{z}.
\end{eqnarray}
The fact that $(\cdot/r\widetilde{r})^{1/2}\mathscr{N}_2({\cdot}/{r\widetilde{r}})$ is bounded guided to
\begin{equation*}
|\textnormal{III}_5|\le \frac{C}{t^{3/2}}\int_{\Omega}e^{-\frac{(r-\widetilde{r})^2+(z-\widetilde{z})^2}{4t}}|f^{r}(\widetilde{r},\widetilde{z})|d\widetilde{r}d\widetilde{z}.
\end{equation*}
By pluging the last estimate in \eqref{Est:02-2} and combine it with \eqref{Est:01-1}, it follows
\begin{equation*}
|\Ss_2(t)\Div f|\le  \frac{C}{t^{3/2}}\int_{\Omega}e^{-\frac{(r-\widetilde{r})^2+(z-\widetilde{z})^2}{5t}}\big(|f^{r}(\widetilde{r},\widetilde{z})|+|f^{z}(\widetilde{r},\widetilde{z})|\big)d\widetilde{r}d\widetilde{z}.
\end{equation*}
Then a new use of Young's inequality leading to the result. 

\hspace{0.7cm}To close our claim, it remains to establish that $\RR_{+}\ni t\mapsto \Ss_1(t)$ (resp. $\RR_{+}\ni t\mapsto \Ss_2(t)$) is continuous on $L^p(\Omega)$ (resp. on $L^p(\Omega)$). We restrict ourselves only for $(\Ss_1(t))_{t\geq0}$. Let $\omega_0\in L^p(\Omega)$ and define its extension on $\RR^2$ by $\widetilde{\omega}_0$ which equal to $0$ outside of $\Omega$. Thus, in view  the change of variables $\widetilde{r}=r+\sqrt{t}\vartheta$ and $\widetilde{z}=z+\sqrt{t}\gamma$, the statement \eqref{exp-exp} takes the form 
\begin{equation*}
(\Ss_1(t)\omega_{0})(r,z)=\frac{1}{4\pi}\int_{\RR^2}\bigg(1+\frac{\sqrt{t}\vartheta}{r}\bigg)^{1/2}\mathscr{N}_1\bigg(\frac{t}{r(r+\sqrt{t}\vartheta)}\bigg)e^{-\frac{\vartheta^2+\gamma^2}{4}}\widetilde{\omega}_{0}(r+\sqrt{t}\vartheta, r+\sqrt{t}\gamma)d\vartheta d\gamma.
\end{equation*}
The fact $$\frac{1}{4\pi}\int_{\RR^2}e^{-\frac{\vartheta^2+\gamma^2}{4}}d\vartheta d\gamma=1,$$ leading to 
\begin{equation}\label{difference}
\Ss_1(t)\omega_{0}(r,z)-\omega_{0}(r,z)=\frac{1}{4\pi}\int_{\RR^2}e^{-\frac{\vartheta^2+\gamma^2}{4}}\Upsilon(t,r,z,\vartheta,\gamma)d\vartheta d\gamma,
\end{equation}
where 
\begin{equation*}
\Upsilon(t,r,z,\vartheta,\gamma)=\bigg(1+\frac{\sqrt{t}\vartheta}{r}\bigg)^{1/2}\mathscr{N}_1\bigg(\frac{t}{r(r+\sqrt{t}\vartheta)}\bigg)\widetilde{\omega}_{0}(r+\sqrt{t}\vartheta, r+\sqrt{t}\gamma)-\widetilde{\omega}_0(r,z).
\end{equation*}
Taking the $L^p-$estimate of \eqref{difference}, then with the aid of the following Minkowski's integral formula in general case
\begin{equation*}
\bigg(\int_{X_1}\Big(\int_{X_2}F(x_1,x_2)d\lambda_2(x_2)\Big)^{p}d\lambda_1(x_1)\bigg)^{1/p}\le \int_{X_2}\bigg(\int_{X_1}F(x_1,x_2)^pd\lambda_1(x_1)\bigg)^{1/p}d\lambda_2(x_2),
\end{equation*}
one obtains for $p\in[1,\infty)$ that
\begin{equation*}
\|\Ss_1(t)\omega_{0}(r,z)-\omega_{0}(r,z)\|_{L^p(\Omega)}\le \frac{1}{4\pi}\int_{\RR^2} e^{-\frac{\vartheta^2+\gamma^2}{4}}\|\Upsilon(t,r,z,\vartheta,\gamma)\|_{L^p(\Omega)}d\vartheta d\gamma. 
\end{equation*}
Now, we must establish that $\|\Upsilon(t,r,z,\vartheta,\gamma)\|_{L^p(\Omega)}\rightarrow 0$ as $t\uparrow0$. To do this, let $r>0$ and $r+\sqrt{t}\vartheta>0$. Writting
\begin{eqnarray*}
\bigg(1+\sqrt{t}\frac{\vartheta}{r}\bigg)^{1/2}\mathscr{N}_1\bigg(\frac{t}{r(r+\sqrt{t}\vartheta)}\bigg)&=&\Big(\frac{\widetilde{r}}{r}\Big)^{1/2}\mathscr{N}_1\Big(\frac{t}{r\widetilde{r}}\Big)\\
&\le & C\frac{|r-\widetilde{r}|}{\sqrt{t}}\le C(1+|\vartheta|).
\end{eqnarray*}
Therefore
\begin{eqnarray*}
\|\Upsilon(t,\cdot,\cdot,\vartheta,\gamma)\|_{L^p(\Omega)}&\le &C(1+|\vartheta|)\big(\|\omega_0(\cdot+\sqrt{t}\vartheta,\cdot+\sqrt{t}\gamma)\|_{L^p(\Omega)}+\|\omega_0\|_{L^p(\Omega)}\big)\\
&\le & C(1+|\vartheta|)\|\omega_0\|_{L^p(\Omega)}.
\end{eqnarray*}
On the other hand, it is clear to verify that $\Big(1+\frac{\sqrt{t}\vartheta}{r}\Big)^{1/2}\mathscr{N}_1\Big(\frac{t}{r(r+\sqrt{t}\vartheta)}\Big)$ goes to $1$ as $t\uparrow 0$. Thus, Lebesgue's dominated convergence asserts for $(\vartheta,\gamma)\in\RR^2$ that $\|\Upsilon(t,r,z,\vartheta,\gamma)\|_{L^p(\Omega)}\rightarrow0$ when $t\uparrow 0$. A new use of Lebesgue's dominated convergence, we finally deduce 
\begin{equation}\label{free-part}
\lim_{t\uparrow0}\|\Ss_1(t)\omega_{0}(r,z)-\omega_{0}(r,z)\|_{L^p(\Omega)}\rightarrow 0,
\end{equation}
which accomplished the proof.
\end{proof}
\hspace{0.7cm}In the spirit of Proposition 3.5 in \cite{Gallay-Sverak}, another weighted estimates for the linear semigroup \eqref{exp-exp} are shown in the following proposition, the proof of which can be done by the same reasoning as in the previous proposition, 
\begin{prop}\label{Prop-W-E-1} Let $ 1\leq p \leq q \leq \infty$, $i\in \{ 1,2\}$ and $(\alpha,\beta)\in[-1,2]$, with $\alpha\leq \beta$. Assume that $r^{\beta}f\in L^p(\Omega)$, then 
\begin{equation}\label{Eq:1-W-est-S1}
\|r^{\alpha}\Ss_i(t) f\|_{L^q(\Omega)}\le \frac{C}{t^{\frac1p-\frac1q+\frac{(\beta-\alpha)}{2}}}\|r^{\beta}f\|_{L^p(\Omega)}.
\end{equation}
In addition, if $(\alpha,\beta)\in[-1,1]$, $\alpha\leq \beta$ and $r^{\beta}f\in L^p(\Omega)$, then 
\begin{equation}\label{Eq:2-W-est-S1}
\|r^{\alpha}\Ss_i(t)\Div_{\star} f \|_{L^q(\Omega)}\le \frac{C}{t^{\frac12+\frac1p-\frac1q+\frac{(\beta-\alpha)}{2}}}\|r^{\beta}f\|_{L^p(\Omega)}.
\end{equation}
\end{prop}
\hspace{0.7cm}We end this section by recalling the following classical estimate on the heat kernel in dimension three, the proof of which is left to the reader.
\begin{prop}\label{3D-heat kernel} Let $ 1\leq p \leq q \leq \infty$. Assume that $ f\in L^p(\RR^3)$, then 
\begin{equation}\label{Eq-3D-heat kernel}
\| \Ss_2(t) f\|_{L^q(\RR^3)}\le \frac{C}{t^{\frac{3}{2}(\frac{1}{p}- \frac{1}{q}) } }\| f\|_{L^p(\RR^3)}.
\end{equation}
\end{prop}

\section{Local existence of solutions}\label{local}
We will explore the aformentioned results and some preparatory topics in the previous sections, we shall scrutinize the local well-posedness issue for the system \eqref{VD}. For this reason, we rewrite it in view of the divergence-free condition in the following form  
\begin{equation}\label{nonlin-prob}
\left\{ \begin{array}{ll} 
\partial_{t}\omega_{\theta}+\Div_{\star}(v\omega_{\theta})=\Big(\partial^{2}_{r}+\partial^{2}_{z}+\frac{1}{r}\partial_{r}-\frac{1}{r^2}\Big)\omega_{\theta}-\partial_r\rho & \textrm{if $(t,r,z)\in \RR_+\times\Omega$,}\\
\partial_{t}\rho+\Div(v\rho)-\Delta \rho=0 & \textrm{if $(t,r,z)\in \RR_+\times\Omega$,}\\ 
(\omega_\theta,{\rho})_{| t=0}=({\omega}_0,{\rho}_0).  \end{array} \right.
\end{equation}

\hspace{0.7cm}The direct treatment of the local well-posedness topic for \eqref{nonlin-prob} in the spirit of \cite{Gallay-Sverak} for initial data $(\omega_{0},\rho_0)$ in the critical space $L^1(\Omega)\times L^1(\RR^3)$ contributes many technical difficulties. This motivates to add the following new unknown $\widetilde{\rho} \triangleq r\rho$ which solves
\begin{equation}\label{tilde-rho}
\partial_{t}\widetilde{\rho}+\Div_{\star}(v\widetilde{\rho})=\Big(\partial^{2}_{r}+\partial^{2}_{z}+\frac{1}{r}\partial_{r}-\frac{1}{r^2}\Big)\widetilde{\rho}-2\partial_r\rho.
\end{equation}
We remark that $\widetilde\rho$ satisfies the same equation as $\omega_{\theta}$ with additional source term and their variations are in $\Omega$.

\hspace{0.7cm}To achieve our topic we will handle with the following equivalent integral formulation.
\begin{equation}\label{int-equation}
\left\{\begin{array}{ll}
\omega_{\theta}(t)=\Ss_1(t)\omega_{0}-\int_{0}^{t}\Ss_1(t-\tau)\Div_{\star}\big(v(\tau)\omega_{\theta}(\tau)\big)d\tau-\int_{0}^{t}\Ss_1(t-\tau)\partial_{r}\rho(\tau) d\tau &\\
\widetilde{\rho}(t)=\Ss_1(t)\widetilde{\rho}_{0}-\int_{0}^{t}\Ss_1(t-\tau)\Div_{\star}\big(v(\tau)\widetilde{\rho}(\tau)\big)d\tau-2\int_{0}^{t}\Ss_1(t-\tau)\partial_{r}\rho(\tau) d\tau &\\
\rho(t)=\Ss_2(t)\rho_0-\int_{0}^{t}\Ss_2(t-\tau)\Div\big(v(\tau)\rho(\tau)\big)d\tau.
\end{array}
\right.
\end{equation}
In order to analyze the above system, we will be working in the following Banach spaces.
\begin{equation*}
X_{T}=\Big\{f\in C^{0}\big((0,T] ,L^{4/3}(\Omega)\big): \Vert f\Vert_{X_{T}}<\infty \Big\},
\end{equation*}
\begin{equation*}
Y_{T}=\Big\{g\in C^{0}\big((0,T] ,L^{4/3}(\Omega)\big): \Vert  g\Vert_{Y_{T}}<\infty \Big\},
\end{equation*}
\begin{equation*}
Z_{T}=\Big\{h\in C^{0}\big((0,T] ,L^{4/3}(\RR^3)\big): \Vert h\Vert_{Z_{T}}<\infty \Big\},
\end{equation*}
equipped with the following norms
\begin{equation*}
\Vert f\Vert _{X_{T}}=\underset{0<t\leq T}{\sup}t^{{1}/{4}}\Vert  f(t)\Vert _{L^{4/3}(\Omega)},\; \Vert g\Vert _{Y_{T}}=\underset{0<t\leq T}{\sup}t^{{1}/{4}}\Vert rg(t)\Vert _{L^{4/3}(\Omega)},\;\Vert h\Vert _{Z_{T}}=\underset{0<t\leq T}{\sup}t^{{3}/{8}}\Vert  h(t)\Vert _{L^{4/3}(\RR^3)}.
\end{equation*}

\hspace{0.7cm}Now, our task is to prove the following result.
 
\begin{prop}\label{local-exist-unique} Let $(\omega_{0},\rho_0)\in L^1(\Omega)\times L^1(\RR^3)$, then there exists $T=T(\omega_{0},\rho_0)$ such that \eqref{int-equation} admits a unique local solution satisfying 
\begin{equation}
(\omega_{\theta},\rho)\in C\big((0,T];X_T\big)\times C\big((0,T];Y_T\cap Z_T\big). 
\end{equation}
\end{prop}

\begin{proof} We will proceed by the fixed point theorem in the product space $\mathscr{X}_T=X_T\times X_T\times Z_T$ equipped by the norm 
$$ 
\|(\omega_\theta, \widetilde{\rho}, \rho) \|_{\mathcal{X}_T} \triangleq \|\omega_\theta \|_{X_T}+\|\widetilde{\rho}  \|_{X_T}+\|\rho \|_{Z_T}.
$$
Notice that by definition, we have 
\begin{equation*}
\|\widetilde{\rho}\|_{X_T}=\|\rho\|_{Y_T}.
\end{equation*}
For $t\ge0$, define the free part $(\omega_{\lin}(t),\widetilde{\rho}_{\lin}(t),\rho_{\lin}(t))=\big(\Ss_1(t)\omega_0,\Ss_1(t)(r\rho_0),\Ss_2(t)\rho_0\big)$, where $\big(\Ss_1(t),\Ss_2(t)\big)$ is given in Proposition \ref{S1,S2}. In accordance with the {\bf(i)}-Proposition \ref{P,S1,S2}, it is not difficult to check that for $(\omega_0,\rho_0)\in L^1(\Omega)\times L^1(\RR^3)$, we have for $T>0$
\begin{equation}\label{initial1}
\sup_{0<t\leq T}t^{{1}/{4}}\|\omega_{\lin}(t)\|_{L^{\frac43}(\Omega)}\leq C\|\omega_0\|_{L^1(\Omega)},
\end{equation}
and
\begin{equation}\label{initial2}
\sup_{0<t\leq T}t^{{1}/{4}}\|\widetilde{\rho} _{\lin}(t)\|_{L^{\frac43}(\Omega)}\leq C\|r\rho_0\|_{L^1(\Omega)} = C\| \rho_0\|_{L^1(\RR^3)}.
\end{equation}
On the other hand, the fact that
\begin{equation*}
\|\rho_{\lin}(t)\|_{L^{\frac43}(\RR^3)}= \|r^\frac{3}{4}\rho_{\lin}(t)\|_{L^{\frac43}(\Omega)}
\end{equation*}
together with \eqref{Eq:1-W-est-S1} stated in Proposition \ref{Prop-W-E-1}, we further get
\begin{equation}\label{initial3}
\sup_{0<t\leq T}t^{{3}/{8}}\|\rho_{\lin}(t)\|_{L^{\frac43}(\RR^3)} \leq  C\|r\rho_0\|_{L^1(\Omega)}=C\|\rho_0\|_{L^1 (\RR^3)}.
\end{equation}
Combining \eqref{initial1}, \eqref{initial2} and \eqref{initial3} to obtain that $(\omega_{\lin},\widetilde{\rho}_{\lin},\rho_{\lin})\in\mathscr{X}_T$.

\hspace{0.7cm}Next, define the following quantity which will be useful in the contraction.
\begin{equation}\label{norm-lin}
\Lambda(\omega_0,\rho_0,T)= C\|(\omega_\lin,\widetilde{\rho}_\lin,  \rho_\lin)\|_{\mathscr{X}_T}.
\end{equation}
We claim that $\Lambda(\omega_0,\rho_0,T)\rightarrow0$ when $T\uparrow0$. To do this, we employ the fact $(L^{4/3}(\Omega)\cap L^1(\Omega))\times (L^{4/3}(\RR^3)\cap L^1(\RR^3))$ is a dense space in $L^1(\Omega)\times L^1(\RR^3)$. Then for every $\EE>0$ and every $(\omega_0,\rho_0)\in L^1(\Omega)\times L^1(\RR^3)$ there exists $(\phi,\psi)\in (L^{4/3}(\Omega)\cap L^1(\Omega))\times (L^{4/3}(\RR^3)\cap L^1(\RR^3))$ such that $$\|(\omega_0,\rho_0)-(\phi,\psi)\|_{L^1(\Omega)\times L^1(\RR^3)}<\EE.$$ On account of {\bf(i)}-Proposition \ref{P,S1,S2} we write
\begin{eqnarray*}
\|\omega_{\lin}(t)\|_{L^{4/3}(\Omega)}&=&\|\Ss_1(t)(\omega_0-\phi+\phi)\|_{L^{4/3}(\Omega)}\\
&\leq&\|\Ss_1(t)(\omega_0-\phi)\|_{L^{4/3}(\Omega)}+\|\Ss_1(t)\phi\|_{L^{4/3}(\Omega)}\\
&\leq& \frac{C}{t^{1/4}}\|\omega_0-\phi\|_{L^1(\Omega)}+C\|\phi\|_{L^{4/3}(\Omega)\cap L^1(\Omega)}.
\end{eqnarray*}
Multiply the both sides by $t^{1/4}$ and taking the supremum over $(0,T]$ to get 
\begin{eqnarray*}
\sup_{0<t\leq T}t^{1/4}\|\omega_{\lin}(t)\|_{L^{4/3}(\Omega)}&\leq &C\|\omega_0-\phi\|_{L^1(\Omega)}+C T^{1/4}\|\phi\|_{L^{4/3}(\Omega)\cap L^1(\Omega)}\\
&\le & C\EE+C T^{1/4}\|\phi\|_{L^{4/3}(\Omega)\cap L^1(\Omega)}.
\end{eqnarray*}
Thus, by setting 
\begin{equation}\label{C-0}
C_0(\omega_0,T)=\sup_{0<t\leq T}t^{1/4}\|\omega_{\lin}(t)\|_{L^{4/3}(\Omega)},
\end{equation} 
and let $T$ (resp. $\EE$) goes to $0$, one deduces
\begin{equation}\label{lim1}
\lim_{T\uparrow0}C_0(\omega_0,T)=0.
\end{equation}
By the same reasoning as above, it holds
\begin{equation*}
\sup_{0<t\leq T}t^{1/4}\|\widetilde{\rho}_{\lin}(t)\|_{L^{4/3}(\Omega)}\leq C\EE+C T^{1/4}\|\phi\|_{L^{4/3}(\Omega)\cap L^1(\Omega)},
\end{equation*}
with 
\begin{equation}\label{C-1}
C_1(\widetilde{\rho}_0,T)=\sup_{0<t\leq T}t^{1/4}\|\widetilde{\rho}_{\lin}(t)\|_{L^{4/3}(\Omega)}.
\end{equation}
Likewise
\begin{equation}\label{lim10}
\lim_{T\uparrow0}C_1(\widetilde{\rho}_0,T)=0.
\end{equation}
For $\rho_{\lin}$, a new use of Propositions \ref{Prop-W-E-1} and \ref{3D-heat kernel} yield
\begin{eqnarray*}
\|\rho_{\lin}(t)\|_{L^{4/3}(\RR^3)}&=&\|\Ss_2(t)(\rho_0-\psi+\psi)\|_{L^{4/3}(\RR^3)}\\
&\leq&\|\Ss_2(t)(\rho_0-\psi))\|_{L^{4/3}(\RR^3)}+\|\Ss_2(t)\psi\|_{L^{4/3}(\RR^3)}\\
&\leq& \frac{C}{t^{3/8}}\|  \rho_0-\psi \|_{L^1(\RR^3)}+ \| \psi\|_{L^{4/3}(\RR^3)} 
\end{eqnarray*}
Now, we multiply the both sides by $t^{3/8}$ and taking the supremum over $(0,T]$ to deduce
 \begin{eqnarray}
\sup_{0<t\leq T}t^{3/8}\|\rho_{\lin}(t)\|_{L^{4/3}(\RR^3)}&\leq& C\|\rho-\psi\|_{L^1(\RR^3)}+C T^{3/8}\|\psi\|_{L^{4/3}\cap L^1(\RR^3)}\\
\nonumber&\le & C\EE+C T^{3/8}\|\psi\|_{L^{4/3}\cap L^1(\RR^3)}.
\end{eqnarray}
Similarly, by putting 
\begin{equation}\label{C-2}
C_2(\rho_0,T)=\sup_{0<t\leq T}t^{3/8}\|\rho_{\lin}(t)\|_{L^{4/3}(\RR^3)},
\end{equation}
we shall obtain that  
\begin{equation}\label{lim2}
\lim_{T\uparrow0}C_2(\rho_0,T)=0.
\end{equation}
Collecting \eqref{lim1}, \eqref{lim10} and \eqref{lim2}, so that by \eqref{norm-lin}, we end up with

\begin{equation*}\label{limit-initial data}
\lim_{T\uparrow 0}\Lambda(\omega_0,\rho_0,T)=0.
\end{equation*}
\hspace{0.7cm}Now, we are ready to contract the integral formulation \eqref{int-equation} in $\mathscr{X}_T$. Doing so, define for $(\omega_{\theta},\widetilde{\rho},\rho)\in\mathscr{X}_T$ the map 
$$
(0,T]\ni t\mapsto\mathscr{T}(t)(\omega_{\theta},\widetilde{\rho},\rho)\in L^{4/3}(\Omega)\times L^{4/3}(\Omega)\times L^{4/3}(\RR^3)
$$ 
by
\begin{equation}\label{T-0}
\mathcal{T}(t)(\omega_{\theta},\widetilde{\rho},\rho)=\begin{pmatrix}
\int_{0}^{t}\Ss_1(t-\tau)\Div_{\star}\big(v(\tau)\omega_{\theta}(\tau)\big)d\tau+\int_{0}^{t}\Ss_1(t-\tau)\partial_{r}\rho(\tau) d\tau \\
\int_{0}^{t}\Ss_1(t-\tau)\Div_{\star}\big(v(\tau)\widetilde{\rho}(\tau)\big)d\tau+2\int_{0}^{t}\Ss_1(t-\tau)\partial_{r}\rho(\tau) d\tau \\
\int_{0}^{t}\Ss_2(t-\tau)\Div\big(v(\tau)\rho(\tau)\big)d\tau
\end{pmatrix}.
\end{equation}
We aim at estimating $\mathscr{T}(t)(\omega_{\theta},\widetilde{\rho},\rho)$  in $L^{4/3}(\Omega)\times L^{4/3}(\Omega)\times L^{4/3}(\RR^3)$. Due to the similarity of the first two lines of \eqref{T-0}, we will restrict ourselves to analyse the first and the third ones.
For $\int_{0}^{t}\Ss_1(t-\tau)\Div_{\star}\big(v(\tau)\omega_{\theta}(\tau)\big)d\tau $, we employ \eqref{S(t)} in Proposition \ref{P,S1,S2} and H\"older's inequality with respect to time to obtain
\begin{align*}
\| \int_{0}^{t}\Ss_1(t-\tau)\Div_{\star}\big(v(\tau)\omega_{\theta}(\tau)\big)d\tau \| _{L^\frac{4}{3}(\Omega)}&\lesssim   \int_{0}^{t} \frac{1}{(t-\tau)^{\frac{1}{2}+ 1-\frac{3}{4}}} \|v(\tau)\omega_{\theta}(\tau)  \| _{L^1(\Omega)}d\tau\\
&\lesssim \int_{0}^{t} \frac{1}{(t-\tau)^{\frac{3}{4}}} \|v(\tau)\|_{L^4(\Omega)} \|\omega_{\theta}(\tau) \| _{L^\frac{4}{3}(\Omega)}d\tau.
\end{align*}
Thanks to \eqref{estimate-v-omega}, it follows that 
\begin{align*}
\| \int_{0}^{t}\Ss_1(t-\tau)\Div_{\star}\big(v(\tau)\omega_{\theta}(\tau)\big)d\tau \| _{L^\frac{4}{3}(\Omega)}&\lesssim \int_{0}^{t} \frac{1}{(t-\tau)^{\frac{3}{4}}}   \|\omega_{\theta}(\tau)  \| _{L^\frac{4}{3}(\Omega)}^2d\tau\\
&\lesssim \int_{0}^{t} \frac{d\tau}{(t-s)^{\frac{3}{4}} \tau^{\frac{1}{2}}} \| \omega_\theta \|_{X_T}^2\\
&\lesssim  t^{-\frac{1}{4}}  \| \omega_\theta \|_{X_T}^2.
\end{align*}
We show next how to estimate $\int_{0}^{t}\Ss_1(t-\tau)\partial_{r}\rho(\tau) d\tau  $ in $L^\frac{4}{3}(\Omega)$. In view of Proposition \ref{Prop-W-E-1} for $\alpha=0$ and $\beta=\frac34$, we get
\begin{align*}
\| \int_{0}^{t}\Ss_1(t-\tau)\partial_{r}\rho(\tau) d\tau \| _{L^\frac{4}{3}(\Omega)} &\lesssim \int_{0}^{t} \frac{1}{(t-\tau) ^{ \frac{1}{2}+ \frac{\frac{3}{4} - 0}{2}}} \| r^\frac{3}{4} \rho\|_{L^\frac{4}{3}(\Omega)} d\tau \\
&\lesssim \int_{0}^{t} \frac{1}{(t-\tau) ^{\frac{7}{8} }} \|   \rho\|_{L^\frac{4}{3}(\RR^3)} d\tau\\
&\lesssim \int_{0}^{t} \frac{d\tau}{(t-\tau) ^{\frac{7}{8} } \tau ^{\frac{3}{8}}} \| \rho\|_{ Z_T }\\
&\lesssim t^{-\frac{1}{4}} \| \rho\|_{ Z_T }.
\end{align*}
The above estimates combined with \eqref{norm-lin} provide the following inequality
\begin{equation}\label{omega-XT..}
\|\omega_\theta \|_{X_T} \leq \Lambda(\omega_0, \rho_0,T) + C\| \omega_\theta  \|_{X_T}^2 + C\| \rho  \|_{Z_T}.
\end{equation}
As explained above, the estimate of $\int_{0}^{t}\Ss_1(t-\tau)\Div_{\star}\big(v(\tau)\widetilde{\rho}(\tau)\big)d\tau$ can be done along the same lines, so we have
\begin{equation} 
\|   \int_{0}^{t}\Ss_1(t-\tau)\Div_{\star}\big(v(\tau)\widetilde{\rho}(\tau)\big)d\tau\| _{L^\frac{4}{3}(\Omega)} \lesssim  t^{-\frac{1}{4}}  \| \omega_\theta \|_{X_T}\| \widetilde{\rho}\|_{X_T},
\end{equation}
we deduce that
\begin{equation}\label{rhotilde-XT}
\|\widetilde{\rho} \|_{X_T} \leq \Lambda(\omega_0, \rho_0,T) + C\| \omega_\theta  \|_{X_T}\|\widetilde{\rho} \|_{X_T} + C\| \rho  \|_{Z_T}.
\end{equation}
Let us move to estimate the last line in \eqref{T-0}. Under the remark $\Div(v\rho) =  \frac{v^r}{r} \rho + \Div_{\star}(v\rho)$, we write 
\begin{eqnarray}\label{T-70}
\int_{0}^{t}\|\Ss_2(t-\tau)\Div\big(v(\tau)\rho(\tau)\big)\|_{L^{4/3}(\RR^3)}d\tau&=&\int_{0}^{t}\Big\|\Ss_2(t-\tau)\Big(\frac{v^r(\tau)}{r}\rho(\tau)\Big)\Big\|_{L^{4/3}(\RR^3)}d\tau\\
\nonumber&&+\int_{0}^{t}\|\Ss_2(t-\tau)\Div_{\star}(v(\tau)\rho(\tau))\|_{L^{4/3}(\RR^3)}d\tau
\end{eqnarray}
So, for the first term, we shall apply \eqref{Eq:1-W-est-S1} stated in Proposition \ref{Prop-W-E-1} for $\alpha=\frac34$ and $\beta=2$ to get
\begin{eqnarray*}
\int_{0}^{t}\Big\|\Ss_2(t-\tau)\Big(\frac{v^r(\tau)}{r}\rho(\tau)\Big)\Big\|_{L^{4/3}(\RR^3)}d\tau&=&\int_{0}^{t}\Big\|r^{3/4}\Ss_2(t-\tau)\Big(\frac{v^r(\tau)}{r}\rho(\tau)\Big)\Big\|_{L^{4/3}(\Omega)}d\tau\\
&\lesssim& \int_0^t \frac{1}{(t-\tau)^{1-{3}/{4} + (2-3/4)/2}} \| v^r (\tau) r\rho(\tau) \|_{L^1(\Omega)} d\tau\\
&\lesssim& \int_0^t \frac{1}{(t-\tau)^{7/8}} \|v^r(\tau)\|_{L^4(\Omega)} \|\widetilde{\rho}(\tau) \|_{L^{4/3}(\Omega)} d\tau\\
&\lesssim& \int_0^t \frac{1}{(t-\tau)^{7/8}\tau^{1/2}} \|\omega_{\theta}\|_{X_T} \|\widetilde{\rho}  \|_{X_T} d\tau\\
&\lesssim& t^{-3/8} \|\omega_{\theta}\|_{X_T} \|\widetilde{\rho}  \|_{X_T}.
\end{eqnarray*}
Therefore
\begin{equation*}
t^{3/8}\int_{0}^{t}\Big\|\Ss_2(t-\tau)\Big(\frac{v^r(\tau)}{r}\rho(\tau)\Big)\Big\|_{L^{4/3}(\RR^3)}d\tau\lesssim \|\omega_{\theta}\|_{X_T} \|\widetilde{\rho}  \|_{X_T}.
\end{equation*}
The second term of the r.h.s. in \eqref{T-70}, will be done by a similar way as above, but we employ \eqref{Eq:2-W-est-S1} in Proposition \ref{Prop-W-E-1} for $\alpha=\frac34$ and $\beta=1$, one may write
\begin{equation*}
t^{3/8}\int_{0}^{t}\|\Ss_2(t-\tau)\Div_{\star}(v(\tau)\rho(\tau))\|_{L^{4/3}(\RR^3)}d\tau\lesssim\|\omega_{\theta}\|_{X_T} \|\widetilde{\rho}  \|_{X_T}.
\end{equation*}
Gathering the last two estimates and insert them in \eqref{T-70}, one has
\begin{equation*}
t^{3/8}\int_{0}^{t}\|\Ss_2(t-\tau)\Div\big(v(\tau)\rho(\tau)\big)\|_{L^{4/3}(\RR^3)}d\tau\lesssim\|\omega_{\theta}\|_{X_T} \|\widetilde{\rho}  \|_{X_T},
\end{equation*}
combined with \eqref{C-2}, it follows
\begin{equation}\label{rho-YT..}
\|\rho \|_{Z_T}\le \Lambda(\omega_{0}, \rho_0,T)+\|\omega_{\theta}\|_{X_T} \|\widetilde{\rho}  \|_{X_T}.
\end{equation}
Collecting \eqref{omega-XT..}, \eqref{rhotilde-XT} and \eqref{rho-YT..}  we finally find the nonlinear system
 \begin{equation}\label{omega-XT}
  \| \omega_\theta\|_{X_T} \leq   \Lambda(\omega_{0}, \rho_0,T)  + C  \| \omega_\theta\|_{X_T}^2 +   \| \rho\|_{Z_T}.
\end{equation}
 \begin{equation}\label{tilde-rho-XT}
  \| \widetilde{\rho}\|_{X_T} \leq   \Lambda(\omega_{0}, \rho_0,T)  +  C \| \omega_\theta\|_{X_T}\| \widetilde{\rho}\|_{X_T}  +   \| \rho\|_{Z_T}.
\end{equation} 
 \begin{equation}\label{rho-YT}
  \| \rho\|_{Z_T} \leq   \Lambda(\omega_{0}, \rho_0,T)  +   C \| \omega_\theta\|_{X_T}\| \widetilde{\rho}\|_{X_T}.
\end{equation}
In order to better justify the contraction argument, let us denote
\begin{equation*}
\mathcal{B}_T(R)\triangleq \{(a, b) \in   X_T\times X_T : \, \|( a,b)\|_{X_T\times X_T} <R \}.
\end{equation*}
and we claim, for $R,T$ sufficiently small, $(\omega_\theta, \widetilde{\rho})\in \mathcal{B}_T(R)$.\\
 By substituting \eqref{rho-YT} into \eqref{omega-XT} and \eqref{tilde-rho-XT}, the contraction argument is satisfyed if
 \begin{equation*}
  3 \Lambda(\omega_{0}, \rho_0,T) + \widetilde{C} R^2<R.
 \end{equation*}

Since $ \Lambda(\omega_{0}, \rho_0,T) \rightarrow 0 $ when $T \uparrow 0$, then an usual argument leads to the existence of $T^{\star}>0$ for which $\| \omega_\theta\|_{X_T} +  \| \widetilde{\rho}\|_{X_T} $ remains bounded by $R$ for all $T<T^{\star}$. Finally by substituting this latest in \eqref{rho-YT} we deduce that $ \| \rho\|_{Y_T} $ remains bounded as well for all $T<T^{\star}$. The local existence and uniqueness follow then from classical fixed-point arguments. For the continuity of the solution, we will postpone the proof after another asymptotic properpties, this completes the proof.
 \end{proof}
 \begin{Rema}\label{remark for T}
 In the light of remark 4.2 from \cite{Gallay-Sverak}, the local time of existence $T$ given by Proposition \ref{local-exist-unique} above can not be bounded from below by using only the norm $\|(\omega_0,\rho_0) \|_{L^1(\Omega) \times L^1(\RR^3)}$. However, in the case where $(\omega_0,\rho_0) \in \big( L^1(\Omega) \big) \times L^1(\RR^3)) \cap \big( L^p(\Omega) \times L^p(\RR^3)\big) $, for some $p>1$, it is easy to explicitely provide a lower bound on $T$ from an upper bound of $\|(\omega_0,\rho_0) \|_{L^p(\Omega) \times L^p(\RR^3)}$ by making use of Propositions \ref{P,S1,S2}, \ref{Prop-W-E-1}, and \ref{3D-heat kernel}.
 \end{Rema}
\hspace{0.7cm}We supply the above local well-posedness result by the following properties of the solution often constructed in the previous part. Especially, we will prove.  
 \begin{prop}\label{properties-decay-Lp}
 For any $p\in (1,\infty)$, we have
 \begin{equation*}
 \underset{t\uparrow 0}{\lim}\, t^{(1-\frac{1}{p})} \| \omega_\theta(t) \|_{L^p(\Omega)} = 0,
 \end{equation*}
  \begin{equation*}
 \underset{t\uparrow 0}{\lim}\, t^{(1-\frac{1}{p})} \| r \rho(t) \|_{L^p(\Omega)} = 0,
 \end{equation*}
  \begin{equation*}
 \underset{t\uparrow 0}{\lim}\, t^{\frac{3}{2} (1-\frac{1}{p})} \| \rho(t) \|_{L^p(\RR^3)} = 0.
 \end{equation*}
 \end{prop}
 \begin{proof} The proof is based principally on a bootstrap argument similar to that of \cite{Gallay-Sverak}. For this aim, we will use the notaions
 \begin{equation*}
 N_p(f,T) \triangleq \underset{0<t\leq T}{\sup} t^{(1-\frac{1}{p})} \|f \|_{L^p(\Omega)},\quad J_p(f,T) \triangleq \underset{0<t\leq T}{\sup}\;t^{\frac{3}{2} (1-\frac{1}{p})} \|  f \|_{L^p(\RR^3)}.
 \end{equation*}
\begin{equation*}
  M_p (f_0, T) \triangleq  \underset{0<t\leq T}{\sup}\;t^{(1-\frac{1}{p})} \|  \Ss_1(t) f_0  \|_{L^p(\Omega)}, \quad F_p  (f_0, T) \triangleq  \underset{0<t\leq T}{\sup}\;t^{\frac{3}{2} (1-\frac{1}{p})} \|  \Ss_2(t)f_0  \|_{L^p(\RR^3)}.
 \end{equation*}
From the properties of the semi-groups $\Ss_1$ and $\Ss_2$, we have for all $p\in (1,\infty]$ 
 \begin{equation}
 \underset{T\uparrow 0}{\lim}\,  M_p (\omega_0, T)=  \underset{T\uparrow 0}{\lim}\,  M_p(r\rho_0, T)  = \underset{T\uparrow 0}{\lim}\, F_p  (\rho _0, T) =  0.
 \end{equation}
In addition, thanks to Proposition \ref{P,S1,S2} the quantities in study are bounded for $p=1$\footnote{The details can be done by following the same approach of Lemma 5.1 from \cite{Gallay-Sverak} concerning the Navier-Stokes equations, but in our case we have an additional term $\partial_r \rho$ which its integral vanishes over $\Omega$.} and from the local existence the desired inequalities hold also for $p=\frac{4}{3}$, thus, by interpolation the proposition in question holds for all $p\in (1,\frac{4}{3}]$. In order to extend it to the other values of $p$ we consider the Duhamel formula \eqref{int-equation}, and we will argue as in the local existence part, thus we omit some steps to make the presentation simpler. In view of Proposition \ref{Prop-W-E-1}, we write
\begin{align*}
\| \omega_\theta(t) \|_{L^p(\Omega)}  \leq \|\Ss_1(t) \omega_0 \|_{L^p(\Omega)}  &+ C \int_0^\frac{t}{2} \frac{\|\omega_\theta \|^2_{L^q(\Omega)}}{ (t-\tau) ^{\frac{2}{q} - \frac{1}{p} }}d\tau + C \int_{ \frac{t}{2}}^t   \frac{\|\omega_\theta(\tau) \|_{L^{q_1}(\Omega)}\|\omega_\theta(\tau) \|_{L^{q_2}(\Omega)}}{ (t-\tau) ^{\frac{1}{q_1} +\frac{1}{q_2} - \frac{1}{p} }}d\tau\\
&+ C \int_0 ^\frac{t}{2} \frac{\| \rho(\tau) \|_{L^\frac{4}{3} (\RR^3)}}{(t-\tau)^{ \frac{1}{2} + \frac{3}{4} - \frac{1}{p} + \frac{3}{8}}} d\tau + C \int_\frac{t}{2}^t  \frac{\| \rho(\tau) \|_{L^p (\RR^3)}}{(t-\tau)^{ \frac{1}{2} + \frac{1}{2p} }}d\tau.
\end{align*} 

Under the conditions
\begin{equation}
\frac{1}{2} \leq  \frac{2}{q} - \frac{1}{p} , \quad \frac{1}{2} \leq \frac{1}{q_1} +\frac{1}{q_2} - \frac{1}{p} <1,
\end{equation}
 we shall obtain
 \begin{equation}\label{N_p-omega-inequa}
 N_p(\omega_\theta,T) \leq M_p (\omega_0,T) + C_{p,q}  N_q(\omega_\theta,T)^2+C_{q_1,q_2}  N_{q_1}(\omega_\theta,T) N_{q_2}(\omega_\theta,T) + C_p J_{\frac{4}{3}}(\rho,T)+ C_p J_p(\rho,T).
 \end{equation}
We recall that $\widetilde{\rho}$ evolves the same equation as $\omega_{\theta}$, so we have
  \begin{equation}\label{N_p-tilde rho inequa}
 N_p(\widetilde{\rho},T) \leq M_p (\widetilde{\rho}_0,T) + C_{p,q}  N_q(\omega_\theta,T)  N_q(\widetilde{\rho} ,T)+C_{q_1,q_2}  N_{q_1}(\omega_\theta,T) N_{q_2}(\widetilde{\rho},T) + C_p J_{\frac{4}{3}}(\rho,T)+ C_p J_p(\rho,T).
 \end{equation}
Finally, to claim similar estimate for $J_p(\rho,T)$, first we write 
 \begin{equation*}
 \| \rho (t) \|_{L^p(\RR^3)}  \leq \|\Ss_2(t) \omega_0 \|_{L^p(\RR^3)}  + C \int_0^\frac{t}{2} \frac{\|\omega_\theta \| _{L^\frac{4}{3} (\Omega)}\|\widetilde{\rho} \| _{L^\frac{4}{3} (\Omega)} }{ (t-\tau) ^{ \frac{1}{2} + 1- \frac{1}{p} + \frac{1-\frac{1}{p}}{2} }}d\tau + C \int_{ \frac{t}{2}}^t   \frac{\|\omega_\theta(\tau) \|_{L^{q_1}(\Omega)}\|\widetilde{\rho} (\tau) \|_{L^{q_2}(\Omega)}}{ (t-\tau) ^{\frac{1}{2} + \frac{1}{\alpha}- \frac{1}{p} + \frac{1- \frac{1}{p}}{2}} }d\tau,
 \end{equation*}
 with 
 \begin{equation*}
 \frac{1}{\alpha} =  \frac{1}{q_1} + \frac{1}{q_2}-\frac{1}{2}.
 \end{equation*}
 Under the additional condition on $p,q_1,q_2$
 \begin{equation*}
  \frac{1}{q_1} + \frac{1}{q_2} - \frac{3}{2p} <\frac{1}{2}
 \end{equation*}
and for $q=\frac{4}{3}$,  we obtain 
\begin{equation}\label{J_p-inequa}
J_p(\rho,T) \leq F_p (\rho_0,T) + C_{p}  N_\frac{4}{3} (\omega_\theta,T)  N_\frac{4}{3} (\widetilde{\rho} ,T)+C_{q_1,q_2}  N_{q_1}(\omega_\theta,T) N_{q_2}(\widetilde{\rho},T).  
\end{equation} 
Plugging \eqref{J_p-inequa} in \eqref{N_p-omega-inequa} and \eqref{N_p-tilde rho inequa} for $q=\frac{4}{3}$, and by denoting 
\begin{equation*}
U_p(T) \triangleq N_p (\omega_\theta,T) + N_p (\widetilde{\rho},T),\quad V_p(T) \triangleq M_p (\omega_0,T)+  M_p (\widetilde{\rho}_0,T) +F_p (\rho_0,T),
\end{equation*}
 we deduce, that
 \begin{equation*}
 U_p(T) \leq C_{p,q_1,q_2} \big(V_p(T) +  U_\frac{4}{3} (T)^2 + J_{\frac{4}{3}}(\rho,T)+  U_{q_1}(T)U_{q_2}(T) \big).
 \end{equation*}
 Now, to cover all the rang $p\in (\frac{4}{3}, \infty)$, we proceed by the following bootstrap algorithm:\\
$\bullet$\; For $q_1=q_2=\frac{4}{3}$ we obviously check that $U_p(T) \rightarrow 0$ as $T\rightarrow 0$ for all $1<p<\frac{3}{2}$.\\
$\bullet$\; Next, by taking $q_1=q_2$ sufficiently close to $\frac{3}{2}$, we obtain the same result, for all $p<\frac{9}{5}$.\\
$\bullet$\; For $q_1=q_2= \frac{8}{5}$, the estimate in question holds for all $p<2$.\\
$\bullet$\; Taking $q_1$ sufficiently close to $2$, the result follows for all $p<\frac{3}{2} q_2$ and for all $q_2<2$.\\
$\bullet$\; Finally, we define the sequence $p_n$ by $p_0= \frac{4}{3}$ and $p_n$ sufficiently close to $\frac{3}{2}p_{n-1} $, by induction, we find that $ p_n $ is sufficiently close to  $(\frac{3}{2})^n p_0$. Hence, letting $n$ goes to $\infty$, we can cover all the rage $p<\infty$, and thus we obtain
\begin{equation*}
U_p(T) \rightarrow 0,\quad T\uparrow 0, \quad \text{for all } p\in (1,\infty).
\end{equation*}
Finally, substituting this latest into \eqref{J_p-inequa}, leads to
\begin{equation}
J_p(T) \rightarrow 0,\quad T\uparrow 0, \quad \text{for all } p\in (1,\infty) .
\end{equation}
This ends the proof of Proposition \ref{properties-decay-Lp}.
 \end{proof}
\hspace{0.7cm}Our last task of this section is to reach the continuity of the solution stated in \eqref{first-result} and \eqref{first-result2} of the main Theorem \ref{First-Th.}. For this aim, we briefly outline the continuity of $\omega_\theta$, the rest of quantities can be treated along the same lines. So, we will show that $$\omega_\theta \in C^0\big([0,T^{\star}); L^p(\Omega) \big), \quad\forall p\in [1,\infty).$$ 
To do so, let $0<t_0\leq t< T^\star $ ($t_0$ close to $0$ for $p=1$), so we have
\begin{equation}\label{continuity-omega}
\omega_\theta(t)-\omega_\theta(t_0) = \big( \Ss_1(t-t_0)-\mathbb{I}\big)\omega_\theta(t_0) - \int_{t_0}^t \Ss_1(t-\tau)\Div_{\star}\big(v(\tau)\omega_\theta(\tau)\big)d\tau - \int_{t_0}^t \Ss_1(t-\tau)\partial_r\rho(\tau) d\tau.
\end{equation}
The first term (free part) is derived by the same manner as in \eqref{free-part}, that is to say,
\begin{equation}\label{Cont.1}
\lim_{ t\uparrow t_0}\|\big( \Ss_1(t-t_0)-1\big)\omega_\theta(t_0,\cdot) \|_{L^p(\Omega)}\rightarrow 0.
\end{equation}
Concerning the second term in the r.h.s of \eqref{continuity-omega}, \eqref{S(t)div} in Proposition \ref{S1,S2} provides
\begin{equation*}
\|\int_{t_0}^t \Ss_1(t-\tau)\Div_{\star}\big(v(\tau)\omega_\theta(\tau)\big)d\tau \|_{L^p(\Omega)}  \lesssim \int _{t_0}^t \frac{1}{(t-\tau)^\frac{1}{2}}\| v(\tau)\|_{L^\infty(\Omega)} \|\omega_\theta(\tau) \|_{L^p(\Omega)}d\tau. 
\end{equation*}
By virtue of the following interpolation estimate, see, Proposition 2.3 in \cite{Gallay-Sverak}, we have for some $1< q_1<2< q_2 <\infty$
\begin{equation*}
\| v(\tau)\|_{L^\infty(\Omega)}  \lesssim \|\omega_\theta(\tau)\|_{L^{q_1}}^{\sigma}\|\omega_\theta(\tau) \|_{L^{q_2}}^{1-\sigma}, \quad \text{with}\quad \sigma= \frac{q_1}{2} \frac{q_2-2}{q_2-q_1}\in(0,1),
\end{equation*}
 one may conclude that
 \begin{eqnarray*}
  &&\|\int_{t_0}^t \Ss_1(t-\tau)\Div_{\star}\big(v(\tau)\omega_\theta(\tau)\big)d\tau \|_{L^p(\Omega)} \\ \lesssim
&&\essup_{\tau \in (t_0,T^{\star})}\big(\|\omega_\theta(\tau) \|_{L^p(\Omega)} \|\omega_\theta(\tau) \|_{L^{q_1}}^{^\sigma}\|\omega_\theta(\tau) \|_{L^{q_2}}^{1-\sigma}\big) \int _{t_0}^t \frac{d\tau}{(t-\tau)^\frac{1}{2}}\\
  &&\lesssim  \essup_{\tau \in (t_0,T^{\star})}\big(\|\omega_\theta(\tau) \|_{L^p(\Omega)} \|\omega_\theta(\tau)\|_{L^{q_1}}^{\sigma}\|\omega_\theta(\tau) \|_{L^{q_2}}^{1-\sigma}\big) (t-t_0)^\frac{1}{2},
 \end{eqnarray*}
which is sufficient to obtain 
\begin{equation}\label{Cont.2}
\lim_{t\uparrow  t_0}\|\int_{t_0}^t \Ss_1(t-\tau)\Div_{\star}\big(v(\tau)\omega_\theta(\tau)\big)d\tau \|_{L^p(\Omega)}=0.
\end{equation}
Let us move to the last term of \eqref{continuity-omega} which we distinguish two cases for $p$. For $p\in (1,\infty)$, \eqref{Eq:2-W-est-S1} stated in Proposition \ref{Prop-W-E-1} for $\alpha=0$ and $\beta=\frac1p$ yielding
\begin{equation}\label{Cont.3}
\|\int_{t_0}^t \Ss_1(t-\tau)\partial_r\rho(\tau) d\tau \|_{L^p(\Omega)} \lesssim \int_{t_0}^t \frac{1}{(t-\tau)^{\frac{1}{2}+ \frac{1}{2p}}} \|r^\frac{1}{p} \rho(\tau) \|_{L^p(\Omega)}d\tau,
\end{equation}
and the fact that $\rho \in L^\infty\big((0,T^*) ;L^p(\RR^3) \big)$ ensures that
\begin{equation*}
\|\int_{t_0}^t \Ss_1(t-\tau)\partial_r\rho(\tau) d\tau \|_{L^p(\Omega)} \lesssim  \essup_{\tau \in (t_0,T^*)}\| \rho(\tau) \|_{L^p(\RR^3)}(t-t_0)^{\frac{1}{2}(1-\frac{1}{p})}
\end{equation*}
combined with \eqref{Cont.3}, one has
\begin{equation}\label{Cont.4}
\lim_{t\uparrow t_0}\|\int_{t_0}^t \Ss_1(t-\tau)\partial_r\rho(\tau) d\tau \|_{L^p(\Omega)}=0.
\end{equation}
For the case $p=1$, we will work with $\widetilde{\Gamma}$ instead of $\omega_{\theta}$ to avoid the source term $\partial_r \rho$. The fact $\| r\rho \|_{L^1(\Omega)} = \| \rho\|_{L^1(\RR^3)} $ leading to 
\begin{equation*}
\| \omega_\theta(t)- \omega_\theta(t_0) \|_{L^1(\Omega)} \leq \| \widetilde{\Gamma}(t)- \widetilde{\Gamma}(t_0) \|_{L^1(\Omega)} + \| \rho(t)- \rho(t_0) \|_{L^1(\RR^3)},
\end{equation*}
so, the continuity of $ \|\omega_\theta(\cdot)\|_{L^p(\Omega)} $ relies then on the continuity of $\|\widetilde{\Gamma}(\cdot)\|_{L^p(\Omega)} $ and $\|  \rho(\cdot) \|_{L^1(\RR^3)}$.\\
On the one hand, seen that the equation of $\tilde{\Gamma}$ governs the same equation to that of $\omega_\theta$, but without the source term $\partial_r \rho$, hence we follow then the same appraoch as above to prove that 
\begin{equation}
\lim_{t\uparrow t_0} \| \widetilde{\Gamma}(t)- \widetilde{\Gamma}(t_0) \|_{L^1(\Omega)} = 0.
\end{equation}
On the other hand, $\rho$ solve a transport-diffusion equation, for which the continuity property is well-known to hold, thus we skip the details. Therefore
\begin{equation*}
\lim_{t\uparrow t_0}\| \omega_\theta(t)- \omega_\theta(t_0) \|_{L^1(\Omega)}=0.
\end{equation*}
Combining the last estimate with \eqref{Cont.1}, \eqref{Cont.2} and \eqref{Cont.4}, we achieve the result. 

\section{Global Existence}  To reach the global existence for the local solution often formulated in sections \ref{local}, we will establish some a priori estimates in Lebesgue spaces. For this target,
\begin{equation*}
(\omega_{\theta},r\rho, \rho)\in C^0\big([0,T];L^p(\Omega)\times L^p(\Omega) \times L^p(\RR^3)\big), \quad p\in [1,\infty), \quad T\in (0,T^*). 
\end{equation*}
be a solution of the integral formulation \eqref{int-equation} and so does $(\omega_\theta,\rho)$ to the differential equation \eqref{nonlin-prob} associted to  initial data $(\omega_0,\rho_0)\in L^1(\Omega)\times L^1(\RR^3)$, where $T^*$ denotes the maximal time of existence. Our basic idea is to couple the system \eqref{nonlin-prob} by introducing the new unknown $\Gamma=\Pi-\frac{\rho}{2}$ following \cite{Hmidi-Rousset} with $\Pi=\frac{\omega_\theta}{r}$. Some familiar computations show that $\Gamma$ obeys   
\begin{equation}\label{Gamma}
\left\{ \begin{array}{ll} 
\partial_{t}\Gamma+v\cdot\nabla \Gamma-(\Delta +\frac{2}{r}\partial_r)\Gamma=0 & \textrm{if $(t,x)\in \RR_+\times\RR^3$,}\\
\Gamma_{| t=0}=\Gamma_0.  
\end{array} \right.
\end{equation}
For our analysis, we need to introduce again the unknown $\widetilde{\Gamma} \triangleq r\Gamma = \omega_\theta - \frac{\widetilde{\rho}}{2}$, which solves
 \begin{equation}\label{tilde-Gamma}
\left\{ \begin{array}{ll} 
\partial_{t}\widetilde{\Gamma}+\Div_{\star}(v \widetilde{\Gamma})-(\Delta -\frac{1}{r^2})\widetilde{\Gamma}=0 & \textrm{if $(t,r,z)\in \RR_+\times\Omega$,}\\
\widetilde{\Gamma}_{| t=0}=\widetilde{\Gamma}_0.  
\end{array} \right.
\end{equation}

The role of the new function $\Gamma$ (resp. $\widetilde{\Gamma})$ for the viscous Boussineq system \eqref{B(mu,kappa)} is the same that $\Pi$ (resp. $\omega_{\theta})$ for the Navier-Stokes equations \eqref{NS(mu)}. For this aim, it is quite natural to treat carefully the properties of $\Gamma$ and $\widetilde{\Gamma}$. 

\hspace{0.7cm}The starting point of our analysis says that $\Gamma$ enjoys the strong maximum principle. We will prove the following.
\begin{prop}\label{maximum-principle} We assume that $\Gamma_0(x_1,x_2,z)>0$ (or, $<0$), then $\Gamma(t,x_1,x_2,z)>0$ (or, $<0$) for any $(x_1,x_2,z)\in\RR^3$ and $t>0$. 
\end{prop} 
\begin{proof} We follow the formalism recently accomplished in \cite{Feng-Sverak}. Up to a regularization of $\Gamma$ by standard method we can achieve the result as follows: we suppose that $\Gamma_0(x_1,x_2,z)>0$ (likewise the case $\Gamma_0(x_1,x_2,z)<0$). Due to the singularity of the term $\frac{2}{r}\partial_r\Gamma$, we can not apply directly the maximum principle. To surmuont this hitch, we can be appropriately interpreted  the term $\Delta+\frac{2}{r}\partial_r$ as the Laplacian in $\RR^5$. Thus we recast \eqref{Gamma} in $]0,\infty[\times\RR^5$ by setting
\begin{equation*}
\overline\Gamma(t,x_1,x_2,x_3,x_4,z)=\Gamma\Big(t,\sqrt{x_1^2+x_2^2+x_3^2+x_4^2},z\Big)
\end{equation*}
and 
\begin{equation*}
\overline{v}(t,x_1,x_2,x_3,x_4,z)=v^r\Big(t,\sqrt{x_1^2+x_2^2+x_3^2+x_4^2},z\Big)\overline{e}_r+v^z\Big(t,\sqrt{x_1^2+x_2^2+x_3^2+x_4^2},z\Big)\overline{e}_z.
\end{equation*}   
Above,
\begin{equation*}
r=\sqrt{x_1^2+x_2^2+x_3^2+x_4^2},\quad \overline{e}_r=\Big(\frac{x_1}{r},\frac{x_2}{r},\frac{x_3}{r},\frac{x_4}{r},0\Big),\quad \overline{e}_z=(0,0,0,0,1)
\end{equation*}
Thus the equation \eqref{Gamma} becomes 
\begin{equation}\label{Gamma2}
\left\{ \begin{array}{ll} 
\partial_{t}\overline\Gamma+\overline v\cdot\nabla_5 \overline\Gamma-\Delta_5\overline\Gamma=0 & \textrm{if $(t,x)\in \RR_+\times\RR^5$,}\\
\overline\Gamma_{| t=0}=\overline\Gamma_0,  
\end{array} \right.
\end{equation}
where $\nabla_5$ and $\Delta_5$ designate the gradient and Laplacian operators over $\RR^5$ respectively. Consequently, by the strong maximum principle for \eqref{Gamma2}, we deduce that
\begin{equation*}
\overline{\Gamma}>0\quad \mbox{in}\quad ]0,\infty[\times \RR^5,
\end{equation*}
which leads to
\begin{equation*}
\Gamma>0 \quad \mbox{in}\quad]0,\infty[\times \RR^3.
\end{equation*}
Thus, the proof is completed.
\end{proof}
\hspace{0.7cm}The second result cares with the classical $L^p-$estimate for $\Gamma$ and showing that $t\mapsto \|\Gamma(t)\|_{L^p(\RR^3)}$ is strictly decreasing function for $p\in[1,\infty]$. We will establish the following.  
\begin{prop}\label{Gamma-LP-prop} Let $v$ be a smooth divergence-free vector field on $\RR^3$ and $\Gamma$ be smooth solution of \eqref{Gamma}. Then the following assertion holds. 
\begin{equation}\label{Gamma-Lp}
\|\Gamma(t)\|_{L^p(\RR^3)}\le \|\Gamma_0\|_{L^p(\RR^3)}, \quad p\in [1,\infty].
\end{equation}
In particular, for $p\in[1,\infty]$ the map $t\mapsto\|\Gamma(t)\|_{L^p(\RR^3)}$ is strictly decreasing.
\end{prop}
\begin{proof} Thanks to the Proposition \ref{maximum-principle}, we can assume that $\Gamma_0>0$, thus we have $\Gamma(t)>0$ for $t\in[0,T]$. We developp an integration by parts and taking into account the $\Gamma-$equation, the fact that $\Div v=0$ and the boundary condition over $\partial\Omega$, one has
\begin{eqnarray}\label{Gamma-p}
\frac{d}{dt}\|\Gamma(t)\|^p_{L^p(\RR^3)}&=&p\int_{\Omega}\partial_{t}\Gamma(t)\Gamma^{p-1}(t)rdrdz\\
\nonumber &=&-p\int_{\Omega}v\cdot(\nabla\Gamma)\Gamma^{p-1}rdrdz+p\int_{\Omega}(\Delta\Gamma)\Gamma^{p-1}rdrdz+2p\int_{\Omega}(\partial_r\Gamma)\Gamma^{p-1}drdz\\
\nonumber&=&-p(p-1)\int_{\Omega}|\nabla\Gamma|^2\Gamma^{p-2}rdrdz+\int_{\Omega}\partial_r\Gamma^{p}drdz\\
\nonumber &=&-p(p-1)\int_{\Omega}|\nabla\Gamma|^2\Gamma^{p-2}rdrdz+\int_{\RR}\Gamma^{p}(t,0,z)\eta_r dz\\
\nonumber &=&-p(p-1)\int_{\Omega}|\nabla\Gamma|^2\Gamma^{p-2}rdrdz-\int_{\RR}\Gamma^{p}(t,0,z)dz< 0.
\end{eqnarray}
where $\eta=(\eta_r,\eta_z)=(-1,0)$ is a outward normal vector over $\Omega$. Thus, integrating in time to obtain the aimed estimate for positive solutions.\\
\hspace{0.7cm}Generally if $\Gamma_0$ changes its sign, we procced as follows: we split $\Gamma(t)=\Gamma^{+}(t)-\Gamma^{-}(t)$, where $\Gamma^{\pm}$ solves the following linear equation with the same velocity
\begin{equation}\label{Gammapm}
\left\{ \begin{array}{ll} 
\partial_{t}\Gamma^{\pm}+v\cdot\nabla \Gamma^{\pm}-(\Delta +\frac{2}{r}\partial_r)\Gamma^{\pm}= 0 & \textrm{if $(t,x)\in \RR_+\times\RR^3$,}\\
\Gamma_{| t=0}^{\pm}=\max(\pm\Gamma_0,0)\ge 0.  
\end{array} \right.
\end{equation}
Arguiging as above to obtain that $\Gamma^{\pm}$ satisfies \eqref{Gamma-Lp}. Thus we have:
\begin{eqnarray}\label{Gamma-result}
\|\Gamma(t)\|_{L^p(\RR^3)}&\le &\|\Gamma^{+}(t)\|_{L^p(\RR^3)}+\|\Gamma^{-}(t)\|_{L^p(\RR^3)}\\
\nonumber &\le & \|\Gamma^{+}_0\|_{L^p(\RR^3)}+\|\Gamma^{-}_0\|_{L^p(\RR^3)}=\|\Gamma_0\|_{L^p(\RR^3)}.
\end{eqnarray}
If $\Gamma_0\ne0$, we distinguish that $\Gamma_0>0$ or $\Gamma_0<0$. For this two cases the last inequality is strict and consequently \eqref{Gamma-Lp} is also strict. Therefore, $t\mapsto\|\Gamma(t)\|_{L^1(\RR^3)}$ is strictly decreasing for $t=0$, and analogously we deduce that is strictly decreasing over $[0,T]$.  
\end{proof}
\hspace{0.7cm}Now, we state a result which deals with the asymptotic behavior of the coupled function $\Gamma$ in Lebegue spaces $L^p(\RR^3)$. Specifically, we have.
\begin{prop}\label{Gamma-3D estimates} Let $\rho_0, \frac{\omega_0}{r}\in L^1(\RR^3)$, then for any smooth solution of \eqref{Gamma} and $1\le p\le\infty$, we have
\begin{equation}\label{Gamma-asymptotique}
\|\Gamma(t)\|_{L^p(\RR^3)}\le \frac{C}{t^{\frac32(1-1/p)}}\|\Gamma_0\|_{L^1(\RR^3)},
\end{equation}
where $\Gamma_0=\Pi_0-\frac{\rho_0}{2}$.
\end{prop}
\begin{proof}  Due to \eqref{Gamma-Lp}, the estimate \eqref{Gamma-asymptotique} is valid for $p = 1$.\\ 
 From the estimate \eqref{Gamma-p} we have for $p=2^n$
\begin{eqnarray}\label{Gamma-pp}
\frac{d}{dt}\|\Gamma(t)\|^p_{L^p(\RR^3)}&=&p\int_{\Omega}\partial_{t}\Gamma(t)\Gamma^{p-1}(t)rdrdz\\
\nonumber&=& -p(p-1)\int_{\Omega}|\nabla\Gamma|^2\Gamma^{p-2}rdrdz-\int_{-\infty}^{\infty}\Gamma^{p}(t,0,z)dz\\
\nonumber &\le & -p(p-1)\int_{\RR^3}|\nabla\Gamma|^2\Gamma^{p-2}dx\\
\nonumber &=& -p(p-1)\int_{\Omega}\bigg|\frac2p\nabla\Gamma^{\frac{p}{2}}\bigg|^2 rdrdz=-\frac{4(p-1)}{p}\int_{\Omega}\big|\nabla\Gamma^{\frac{p}{2}}\big|^2 rdrdz.
\end{eqnarray}
Thanks to the well-known Nash's inequality in general case
\begin{equation}\label{Nash-Ineq}
\int_{\RR^N}|f|^2 dx\le C\bigg(\int_{\RR^N}|\nabla f|^2 dx\bigg)^{1-\gamma}\bigg(\int_{\RR^N}|f| dx\bigg)^{2\gamma},\quad \gamma=\frac{2}{N+2}.
\end{equation}
one obtains for $N=3$
\begin{equation*}
-\frac{d}{dt}\int_{\Omega}\Gamma^{p}(t)rdrdz\ge \frac{4(p-1)}{p}C\bigg(\int_{\Omega}\big|\Gamma^{\frac{p}{2}}\big| rdrdz\bigg)^{-4/3}\bigg(\int_{\Omega}\Gamma^p rdrdz\bigg)^{5/3}.
\end{equation*}
To simplify the presentation, setting $E_{p}(t)=\|\Gamma(t)\|_{\RR^3}^{p}=\int_{\Omega}|\Gamma(t)|^{p}rdrdz$, then the last inequality becomes
\begin{equation}\label{Energy-1}
-\frac{d}{dt} E_{p}(t)\ge \frac{4(p-1)}{p}C E_{p/2}^{-4/3}(t)E_{p}^{5/3}(t)
\end{equation}
We prove \eqref{Gamma-asymptotique} for $p=2^n$ with nonnegative integers $n$ by induction. Assume that \eqref{Gamma-asymptotique} is true for $q=2^k$ with $k\ge0$, and let $p=2^{k+1}$. Combined with \eqref{Energy-1}
\begin{equation*}
-\frac{d}{dt} E_{p}(t)\ge \frac{4(p-1)C}{p} \big(C_q^qt^{-\frac32(q-1)}\|\Gamma_0\|_{L^1(\RR^3)}^q\big)^{-4/3} E_{p}^{5/3}(t).
\end{equation*}
Thus we have 
\begin{eqnarray*}
\frac{3}{2}\frac{d}{dt}\Big(E_p(t)\Big)^{-2/3}=\frac{-\frac{d}{dt} E_{p}(t)}{E_{p}^{5/3}(t)}&\ge & \frac{4(p-1)C}{p} C_q^{-4q/3}\|\Gamma_0\|_{L^1(\RR^3)}^{-4q/3} t^{2(q-1)}\\
\nonumber&=& \frac{4(p-1)C}{p} C_q^{-2p/3}\|\Gamma_0\|_{L^1(\RR^3)}^{-2p/3} t^{(p-2)}.
\end{eqnarray*}
Hence, integrating in time le last inequality yields
\begin{equation*}
E_{p}^{-2/3}(t)\ge E_{p}^{-2/3}(0)+\frac{8C}{3p} C_q^{-2p/3}\|\Gamma_0\|_{L^1(\RR^3)}^{-2p/3}t^{p-1}.
\end{equation*}
After a few easy computations, we derive the following 
\begin{equation*}
\|\Gamma(t)\|_{L^p(\RR^3)}=E_p^{\frac1p}(t)\le\Big(\frac{3p}{8C}\Big)^{\frac{3}{2p}}C_q\|\Gamma_0\|_{L^1(\RR^3)} t^{-3/2(1-1/p)}.
\end{equation*}
By setting $C_p=\Big(\frac{3p}{8C}\Big)^{\frac{3}{2p}}C_q$, then \eqref{Gamma-asymptotique} remains true for $p=2^{k+1}$. Let us observe that
\begin{eqnarray*}
C_p=\Big(\frac{3p}{8C}\Big)^{\frac{3}{2p}}C_q&=& \Big(\frac{3}{8C}\Big)^{\frac{3}{2^{k+2}}}2^{\frac{3(k+1)}{2^{k+2}}}C_{2^k}\\
&\le& \Big(\frac{3}{8C}\Big)^{\frac34\sum_{k\ge0}\frac{1}{2^{k}}}2^{\frac{3}{4}\sum_{k\ge0}\frac{k+1}{2^{k}}}C_1\triangleq C_{\infty} 
\end{eqnarray*}
which means that $C_{\infty}$ is independtly of $p$. Letting $p\rightarrow\infty$, we deduce that
\begin{equation}\label{Gamma-Linfty}
\|\Gamma(t)\|_{L^\infty(\RR^3)}\le C_{\infty}t^{-3/2}\|\Gamma_0\|_{L^1(\RR^3)}.
\end{equation}
For the other values of $p$, we proceed by complex interpolation to get
\begin{equation*}
\|\Gamma(t)\|_{L^p(\RR^3)}\le C\|\Gamma(t)\|_{L^1(\RR^3)}^{1/p}\|\Gamma(t)\|_{L^\infty(\RR^3)}^{1-1/p},
\end{equation*}
combined with \eqref{Gamma-Linfty}, so the proof is completed. 
\end{proof}
\hspace{0.7cm}Next, we recall some a priori estimates for $\rho-$equation in Lebesgue spaces. To be precise, we have.
 
\begin{prop}\label{rho-estimate} Let $\rho_0\in L^1(\RR^3)$ and $p\in [1,\infty]$, then there exists some nonnegative universal constant $C_p>0$ depending only on $p$ such that for any smooth solution of $\rho-$equation in \eqref{VD}, we have
\begin{enumerate}
\item[{\bf(i)}]  $\|\rho(t)\|_{L^p(\RR^3)}\le\|\rho_0\|_{L^p(\RR^3)}$,
\item[{\bf(ii)}] $\|\rho(t)\|_{L^p(\RR^3)}\leq\frac{C_p}{t^{\frac{3}{2}(1-\frac1p)}}\|\rho_0\|_{L^1(\RR^3)}$.
\end{enumerate}
\end{prop}
\begin{proof} {\bf(i)} Can be done by a routine compuations as shown in Proposition \ref{Gamma-LP-prop}, while {\bf(ii)} can be obtained along the same way as Proposition \ref{Gamma-3D estimates}.\\
\hspace{0.7cm}We should mention also that the constant $C_p$ is bounded with respect to $p$ (see the proof of Proposition \ref{Gamma-3D estimates}), and according to the proof of Proposition \ref{Gamma-3D estimates} $C_\infty$ is given by
\begin{equation}\label{C_infty}
C_\infty \triangleq \Big(\frac{3}{8C}\Big)2^{\frac{3}{4}\sum_{k\ge0}\frac{k+1}{2^{k}}}C_1 < \infty
\end{equation}
\end{proof}
\hspace{0.7cm}Now, we will prove another type of estimates for the quantities $\widetilde{\Gamma}, \widetilde{\rho}$ and $\omega_\theta$. Namely, we establish.
\begin{prop}\label{Gamma-estimate} Let $\rho_0, \frac{\omega_0}{r} \in L^1(\RR^3)$ and $ p\in [1,\infty]$, then there exist a nonnegative constants $\widetilde{C}_p, K_p $, depending only on $p$ and the initial data, such that for any smooth solution of \eqref{tilde-Gamma}, \eqref{tilde-rho} and \eqref{nonlin-prob}, we have
\begin{enumerate}
\item[{\bf(i)}] $\|\widetilde{\Gamma}(t)\|_{L^p(\Omega)}\leq\frac{\widetilde{C}_p ( D_0)}{t^{ 1-\frac1p }}$,
\item[{\bf(ii)}] $\|\widetilde{\rho}(t)\|_{L^p(\Omega)}\leq\frac{K_p(D_0)}{t^{ 1-\frac1p }} $,
\item[{\bf(iii)}] $\|\omega _\theta(t)\|_{L^p(\Omega)}\lesssim  \frac{\widetilde{C}_p (D_0) + K_p(D_0)}{t^{ 1-\frac1p }}  $,
\end{enumerate}
where
\begin{equation}\label{D_0}
D_0 = \|(\omega _0,  \rho_0)\|_{L^1(\Omega) \times L^1(\RR^3)}
\end{equation}
and
\begin{equation}\label{tilde C_p asympto}
\sup_{p\in [1,\infty)}\widetilde{C}_p(s) \triangleq \widetilde{C}_\infty(s)< \infty, \quad  \widetilde{C}_p(s) \rightarrow 0, \; \text{as } s\uparrow 0, \quad \forall  p\in [1,\infty].
\end{equation}
\end{prop}
 \begin{proof} Let us point out that {\bf(iii)} is a consequence of {\bf(i)} and {\bf(ii)}. Thus, we shall focus ourselves to prove {\bf(i)} and {\bf(ii)}.
 
{\bf(i)} Due to the similarity of the equation of $\widetilde{\Gamma}$ and the one of $\omega_{\theta}$ for the Navier-Stokes \eqref{NS(mu)} treated in \cite{Gallay-Sverak}, we follow the approach stated in Proposition 5.3 in \cite{Gallay-Sverak}. The key point consists to employ the following estimate, 
 \begin{align}\label{vr/r estimate}
  \Big\| \frac{v^r}{r} \Big\|_{L^\infty(\Omega)}  \lesssim \frac{1}{t}\| (\omega_0,r\rho_0) \|_{L^1(\Omega)}.
 \end{align}
 Indeed, Proposition 2.6 in \cite{Gallay-Sverak} gives 
 \begin{equation*}
 \Big\| \frac{v^r}{r} \Big\|_{L^\infty(\Omega)}  \lesssim \|\omega_\theta \|_{L^1(\Omega)}^\frac{1}{3} \Big\|\frac{\omega_\theta}{r}\Big\|_{L^\infty(\Omega)}^\frac{2}{3},
 \end{equation*}
 by using the fact that
 \begin{equation*}
 \Big\|\frac{\omega_\theta}{r}\Big\|_{L^\infty(\Omega)} = \Big\|\frac{\omega_\theta}{r}\Big\|_{L^\infty(\RR^3)}
 \end{equation*}
 combined with  $\frac{\omega_\theta}{r} = \frac{ \rho}{2} +  \Gamma$, together with Propositions \ref{Gamma-3D estimates} and \ref{rho-estimate}, lead to \eqref{vr/r estimate}. So, the inequality {\bf(i)} follows then by exploring \eqref{vr/r estimate} and repeating the outlines of the proof of Proposition 5.3 from \cite{Gallay-Sverak}, the details are left to the reader. We should only mention that the constant $\tilde{C}_p$ in our proposition is the same as the one from proposition 5.3 in \cite{Gallay-Sverak}, which guaranties \eqref{tilde C_p asympto}.
  
{\bf(ii)} The estimate obviously holds for $p=1$, whereas in the rest of the proof we shall deal with $p>3$. The case $p\in (1,3] $ follows by interpolation. 

\hspace{0.7cm}We multiply the $\widetilde{\rho}-$equation by $ |\widetilde{\rho}| ^ {p-1} $, after some integrations by parts we obtain
 \begin{equation}\label{iv-estimate}
 \frac{1}{p}\frac{d}{dt} \|\widetilde{\rho}(t) \|_{L^p(\Omega)}^p \leq - 4\frac{(p-1)}{p^2} \int_\Omega |\nabla(|\widetilde{\rho}|^\frac{p}{2})|^2drdz + \bigg| - \int_\Omega   \Div_{\star}(v \widetilde{\rho}) |\widetilde{\rho}|^{p-1}drdz  - \int_\Omega \partial_r \rho |\widetilde{\rho}|^{p-1}drdz \bigg|.
 \end{equation}
On the one hand, a straightforward computation give
  \begin{equation*}
  - \int_\Omega \Div_{\star}(v \widetilde{\rho}) |\widetilde{\rho}|^{p-1}drdz = \bigg(1- \frac{1}{p} \bigg) \int_\Omega \frac{v^r}{r} |\widetilde{\rho}|^pdrdz,
  \end{equation*}
then \eqref{vr/r estimate} provides
\begin{equation}\label{tilde-rho-div}
 - \int_\Omega \Div_{\star}(v \widetilde{\rho}) |\widetilde{\rho}|^{p-1}drdz \leq C D_0 \bigg(1- \frac{1}{p} \bigg)  t^{-1} \int_\Omega |\widetilde{\rho}|^pdrdz,
\end{equation}
where $D_0 $ is given by \eqref{D_0}.

On the other hand, the fact $-\partial_r \rho = \frac{-\partial_r \widetilde{\rho}}{r} + \frac{\widetilde{\rho}}{r^2}$ yields
\begin{equation*}
-\int_\Omega \partial_r \rho |\widetilde{\rho}|^{p-1}drdz = \bigg(1- \frac{1}{p} \bigg) \int_\Omega \frac{ |\widetilde{\rho}|^p}{r^2} drdz.
\end{equation*}
Next, let us write
\begin{equation}\label{Est-I1+I2}
\int_\Omega \frac{ |\widetilde{\rho}|^p}{r^2} drdz = \mathcal{I}_1 + \mathcal{I}_2,
\end{equation}
with
\begin{equation*}
\mathcal{I}_1  \triangleq \int_\Omega \frac{ |\widetilde{\rho}|^p}{r^2}{\bf1}_{\{r\leq t^{1/2}\}} drdz ,\quad    \mathcal{I}_2  \triangleq \int_\Omega \frac{ |\widetilde{\rho}|^p}{r^2}{\bf 1}_{\{r> t^{1/2}\}}(r,z) drdz.
\end{equation*}
 For $p>3$ we have, 
  \begin{align*}
   \mathcal{I}_1  = \int_{\Omega} r^{p-3}| \rho|^p{\bf 1}_{\{r\leq t^{1/2}\}} rdrdz &\leq t^\frac{p-3}{2} \| \rho\|_{L^p(\RR^3)}^p.
 \end{align*}
So, by virtue of Proposition \ref{rho-estimate}, we infer that
 \begin{equation}\label{tilde-rho-force-term-1 }
 \mathcal{I}_1 \leq C_p^p t^{-p} G_0^p,
 \end{equation}
 where $G_0 \triangleq \|\rho_0\|_{L^1(\RR^3)}$, and $C_p$ is the constant given by Proposition \ref{rho-estimate}.
 
\hspace{0.7cm}For the term $ \mathcal{I}_2$ an easy computation yields 
  \begin{equation}\label{tilde-rho-force-term-2 }
  \mathcal{I}_1 \leq t^{-1} \int_\Omega |\widetilde{\rho}|^p drdz.
  \end{equation}
Therefore, \eqref{tilde-rho-force-term-1 } and \eqref{tilde-rho-force-term-2 } give rise to
 \begin{equation}\label{tilde-rho-force-term}
 - \int_\Omega \partial_r \rho |\widetilde{\rho}|^{p-1}drdz \leq  \frac{p-1}{p} \bigg(C_p^p t^{-p} G_0^p+ t^{-1} \int_\Omega |\widetilde{\rho}|^p drdz \bigg).
 \end{equation}
 Finally, Nash's inequality allows us to write
 \begin{equation}\label{Nash-tilde-rho}
 \int_{\Omega } |\widetilde{\rho}|^p drdz \lesssim \bigg( \int_\Omega |\nabla(|\widetilde{\rho}|^\frac{p}{2})|^2 drdz \bigg) ^ \frac{1}{2} \bigg(\int_{\Omega} |\widetilde{\rho}|^\frac{p}{2} drdz\bigg).
 \end{equation}
 Since $p>3$, so we have $1< \frac{p}{2}< p$, and then by interpolation method, it happens
 \begin{equation*}
 \bigg(\int_{\Omega} |\widetilde{\rho}|^\frac{p}{2}drdz\bigg)   \lesssim  \bigg(\int_\Omega | \widetilde{\rho} |drdz\bigg) ^\frac{p}{2(p-1)}  \bigg(\int_\Omega  |\widetilde{\rho}|^p drdz\bigg) ^\frac{p-2}{2(p-1)}. 
\end{equation*}  
 Plugging the last inequality in \eqref{Nash-tilde-rho}, it holds
 \begin{equation*}
 \bigg( \int_{\Omega } |\widetilde{\rho}|^p drdz \bigg)^\frac{p}{2(p-1)}\lesssim \bigg( \int_\Omega |\nabla(|\widetilde{\rho}|^\frac{p}{2})|^2 \bigg) ^ \frac{1}{2} \bigg(\int_{\Omega} |\widetilde{\rho}| drdz \bigg) ^ \frac{p}{2(p-1)}.
 \end{equation*}
 Since the inequality we aim to prove holds for $p=1$, accordingly
 \begin{equation}\label{tilde-rho-grad}
 C G_0 ^ {-\frac{p}{p-1}}\bigg( \int_{\Omega } |\widetilde{\rho}|^p drdz \bigg)^\frac{p}{p-1}\leq  \bigg( \int_\Omega |\nabla(|\widetilde{\rho}|^\frac{p}{2})|^2 drdz\bigg).    
 \end{equation}
 Thus, by gathering \eqref{tilde-rho-div}, \eqref{tilde-rho-force-term} and \eqref{tilde-rho-grad} and insert them in \eqref{iv-estimate}, it happens 
 \begin{equation}\label{ine-diff-1}
 f'(t) \lesssim     (p-1)\bigg( -\frac{C}{p} G_0 ^ {-\frac{p}{p-1}}\big( f(t)\big)^\frac{p}{p-1}  + (C D_0+1)t^{-1} f(t)   +   C_p^p t^{-p} G_0^p \bigg),
 \end{equation}
 where, $f(t) \triangleq \int_\Omega |\widetilde{\rho}(t)|^p drdz$.

\hspace{0.7cm} We recall that one may deduce from Proposition \ref{properties-decay-Lp}, for all $p\in [1, \infty)$
 \begin{equation}\label{decay-f}
f(t) \leq e_p(D_0)^p{t^{ -(p-1) }}, \quad \forall\; 0<t<T^{\star},
 \end{equation}
 for some $e_p(D_0)>0$.

In a first step, we will show that $f(t)$ is finite for all $t>0$, then we prove that the decay property \eqref{decay-f} holds as well for all $t>0$, for a suitable non negative constant $K_p(G_0)$. Indeed, the first step is easy, one should remark that \eqref{ine-diff-1} implies
\begin{equation*}
f'(t) \lesssim     (p-1)\Big(  (C D_0+1)t^{-1} f(t)   +   C_p^p t^{-p} G_0^p \Big).
\end{equation*}
Via, Gronwall inequality on $(t_0,t)$, for some $0<t_0<T^{\star} $, we get for all $t>t_0$
\begin{equation}
f(t) \leq \big( f(t_0) +  C_p^p t_0 ^{1-p} \big) \bigg( \frac{t}{t_0}\bigg)^{ (p-1) (CD_0 + 1)},
\end{equation}
which ensures that $f(t)$ is finite for all $t>0$.

\hspace{0.5cm} Now, let us denote
\begin{equation}\label{g-prime}
\widetilde{T} \triangleq \sup\big\{ t>0: \; f(t) < K_p^p(D_0){t^{ -(p-1) }} \big\},
\end{equation} 
where $K_p(D_0)$ will be chosen later, and we will prove that \eqref{decay-f} holds as well for all $t\in [\widetilde{T}, \widetilde{T}+ \varepsilon]$, for some $\varepsilon>0$, this should be enough to contradict the fact that $\widetilde{T}<\infty$, and we shall conclude then that \eqref{decay-f} is true for all $t>0$. If $\widetilde{T}$ is finite then we deduce 
\begin{equation}\label{f(tildeT)}
f(\widetilde{T}) = K_p^p(D_0){\widetilde{T} ^{ -(p-1) }}.
\end{equation}
Now, define $g$ by 
\begin{equation*}
g(t) \triangleq f(t) - K_p^p(D_0){t^{ -(p-1) }}.
\end{equation*}
By virtue of \eqref{ine-diff-1} and \eqref{f(tildeT)}, we find out that
\begin{equation}\label{inequality-g}
g'(\widetilde{T}) \leq \widetilde{T}^{-p}(p-1) \bigg(- \frac{C}{p} G_0^{-\frac{p}{p-1}}  \big( K_p(D_0)\big) ^ \frac{p^2}{p-1} + K_p^p(D_0) +  \Sigma_p( D_0)   \bigg),
\end{equation}
where $\Sigma_p(D_0) = CD_0+1+ C_p^p D_0^p$. Since $\frac{p^2}{p-1}> p$, then if we choose $K_p^p (D_0)$ large enough, in terms of $\Sigma_p(D_0)$ and $ G_0$, we may conclude that
\begin{equation*}
g'(\widetilde{T}) <0,
\end{equation*}
which in particular gives, for $\varepsilon\ll1$ 
\begin{equation*}
g(\widetilde{T}+\varepsilon) \leq g(\widetilde{T}) = 0.
\end{equation*}
This means that \eqref{decay-f} holds for $t= \widetilde{T}+ \varepsilon$, which contracits the fact that $\tilde{T}$ is finite. The choice of $K_p(D_0)$ can be made as 
\begin{equation}\label{K_p}
K_p (D_0) = \max\bigg\{ \bigg( C^{-1} pG_0^\frac{p}{p-1} \big( CD_0+2 + C_p^pD_0^p \big)  \bigg)^\frac{1}{p}, 1 \bigg\}.
\end{equation}
and end with, for all $p>3$
\begin{equation}\label{estimate p>3}
\| \widetilde{\rho}\|_{L^p(\Omega)} \leq K_p(D_0) t^{-(1-\frac{1}{p})}.
\end{equation}
We denoting  
\begin{equation*}
 K_\infty(D_0) \triangleq \underset{p\rightarrow \infty}{\lim} \max\bigg\{ \bigg( C^{-1} pG_0^\frac{p}{p-1} \big( CD_0+2 + C_\infty ^pD_0^p \big)  \bigg)^\frac{1}{p}, 1 \bigg\} = 1 + C_\infty D_0 .
\end{equation*} 
From proposition \ref{rho-estimate}, $C_\infty$ is finite, hence $ K_\infty(D_0)$ is also finite, and since from the definition of $K_\infty(D_0)$, we have 
\begin{equation*}
K_p(D_0) \leq K_\infty(D_0),\quad \forall p\in[1,\infty),
\end{equation*}
 then by letting $p\rightarrow \infty$ in \eqref{estimate p>3}, we end up with
\begin{equation*}
\| \widetilde{\rho}\|_{L^\infty (\Omega)} \leq K_\infty (D_0) t^{-1 }.
\end{equation*}
 \end{proof}
\begin{Rema} As pointed out for the Navier-Stokes equations in Remark 5.4 from \cite{Gallay-Sverak}, for the global existence part in our case, we only need to mention that due to Proposition \ref{Gamma-estimate} (resp. Proposition \ref{rho-estimate}) the $L^p(\Omega)$ norms of $\omega_\theta(t)$ and $r\rho(t)$ (resp. the $L^p(\RR^3)$ of $\rho(t)$) can not blow-up in finite time, hence in view of remark \ref{remark for T}, it turns out that any constructed solution in the previous section is global for positive time, in addition of that, all the assertions \eqref{first-result}$-$\eqref{boundedness2} follow as a consequence of propositions \ref{rho-estimate} and \ref{Gamma-estimate}.
 \end{Rema}
 \begin{Rema}
 To be more precise about the assertions \eqref{first-result} and \eqref{first-result2} in the case of $L^\infty$, remark that the constants $\tilde{C}_p(  D_0  )$ and $K_p(D_0)$ given by the last Proposition above do not blow-up as $p$ goes to infinity, which we do not know whether it holds true or not for the constants that appear in the bootstrap argument in the proof of Proposition \ref{properties-decay-Lp}, this is why we did not say any thing about the $L^\infty$ case in the local existence part. As fact of matter, now while we know that the triplet $\big( \omega_\theta(t_0),r\rho(t_0),\rho(t_0)\big) $ holds to be in $L^\infty(\Omega)\times L^\infty(\Omega) \times L^\infty(\RR^3)$, for all $t_0>0$, we can prove that the map 
 $$ t \mapsto  \big( \omega_\theta(t),r\rho(t ),\rho(t )\big)$$
 is continious with value in $L^\infty(\Omega)\times L^\infty(\Omega) \times L^\infty(\RR^3) $ just by following exactely the same procedure we showed in the case $p\in (1,\infty)$, the details are left to the reader.
 \end{Rema}

 \section*{Acknowledgements}
 This work has been done while the second author is a PhD student at the University of Cote d'Azur-Nice-France, under the supervision of F. Planchon and P. Dreyfuss, in particular the first author would like to thank his supervisers, and the LJAD direction.\\

The third author appreciates Taoufik Hmidi from Rennes 1 University for the fruitful discussions on the subject with the occasion of RAMA 11 organized by Sidi bel Abbess University, Algeria, November 2019.

\end{document}